\documentclass[11pt,a4paper]{article}

\usepackage{amsmath,amsthm, amssymb, latexsym}
\usepackage{amsfonts}
\usepackage[T1]{fontenc}
\usepackage{lmodern}
\usepackage{graphicx}
\usepackage{amsfonts}
\usepackage{multicol}
\usepackage{caption}
\usepackage{color}
\usepackage{graphicx}
\usepackage{amsmath}
\usepackage{mathrsfs}
\usepackage{multicol}
\usepackage{float}
\usepackage{bbold}
\usepackage{color}
\usepackage{subfigure}
\usepackage{url}
\usepackage{tabularx}
\usepackage{xcolor}
\usepackage{grffile}
\usepackage{cancel}
\usepackage[english]{babel}
\usepackage{fullpage}
\usepackage{enumitem}
\usepackage[hidelinks]{hyperref}
\usepackage{array}
\usepackage{tabularx, caption}
\usepackage{cellspace}
 \normalbaselines
\newtheorem{Theorem}{Theorem}[section]
\newtheorem{Definition}[Theorem]{Definition}
\newtheorem{Proposition}[Theorem]{Proposition}
\newtheorem{Lemma}[Theorem]{Lemma}
\newtheorem{Corollary}[Theorem]{Corollary}

\newtheorem{Remark}[Theorem]{Remark}

\newtheorem{Assumption}[Theorem]{Assumption}
\newcommand{\N}{\mathbb{N}}
\newcommand{\Z}{\mathbb{Z}}
\newcommand{\R}{\mathbb{R}}

\newcommand{\norm}[1]{\left\lVert #1 \right\rVert}	
\newcommand{\normabs}[1]{\left| #1 \right|}	
\newcommand{\abs}[1]{\left| #1 \right|}	
\newcommand{\ra}{\rightarrow}
\newcommand{\Lnorm}[1]{{\left\lVert #1 \right\rVert}_{L^2}}

\newcommand\restr[2]{{
  \left.\kern-\nulldelimiterspace 
  #1 
  \vphantom{\big|} 
  \right|_{#2} 
  }}
\allowdisplaybreaks[4]
\newcommand{\mylabel}[2]{#2\def\@currentlabel{#2}\label{#1}}

\begin{document}

\title{Asymptotics of time-varying processes in continuous-time using locally stationary approximations\\
}
\author{Robert Stelzer and Bennet Ströh\footnote{Ulm University, Institute of Mathematical Finance, Helmholtzstra\ss e 18, 89069 Ulm, Germany. Emails: robert.stelzer@uni-ulm.de, bennet.stroeh@uni-ulm.de. }}

\maketitle

\textwidth=160mm \textheight=225mm \parindent=8mm \frenchspacing
\vspace{3mm}

\begin{abstract}
We introduce a general theory on stationary approximations for locally stationary contin\-uous-time processes. Based on the stationary approximation, we use $\theta$-weak dependence to establish laws of large numbers and central limit type results under different observation schemes. Hereditary properties for a large class of finite and infinite memory transformations show the flexibility of the developed theory. Sufficient conditions for the existence of stationary approximations for time-varying L\'evy-driven state space models are derived and compared to existing results in \cite{BS2021}. We conclude with comprehensive results on the asymptotic behavior of the first and second order localized sample moments of time-varying L\'evy-driven state space models.
\end{abstract}


\section{Introduction}
\label{sec1}
In the last years, extensive research on continuous-time processes led to a variety of flexible yet analytically tractable models \cite{B2014, C2018, KLM2004, MS2007, St2010}. Not least because irregularly spaced time series and high frequency data can be handled, continuous-time processes have been successfully applied in areas as diverse as finance, physics, and engineering \cite{BS2001, BKMV2014, BFK2013, LM2004}. Despite this success, many of the established models as the well-known L\'evy-driven state space model (or equivalently CARMA model), are not capable of modeling non-stationary behavior in observed data. \\
In the discrete-time setting, this problem has, for instance, been treated with a rich theory on locally stationary processes. Starting in \cite{D1996}, a consistent framework for locally stationary time-varying models (see \cite{D2012} for a comprehensive overview) has been developed by Dahlhaus and others. Recently, this approach has been generalized in \cite{BDW2020} and \cite{DRW2019}, where the authors introduce a remarkably versatile approach for local stationarity, including non-linear models of finite and infinite memory, which are beyond the classical linear time series models.\\
However, in the continuous-time setting, things look different. Except for results on time-varying Gaussian-driven diffusion models as in \cite{KL2012}, the only definition on local stationarity in continuous-time including non-Gaussian models was recently given in \cite{BS2021}. There, the authors proposed a parametric approach that can be considered as the continuous-time analog of the original approach from Dahlhaus \cite{D1996}. Unfortunately, there are no statistical results available based on this concept of local stationarity.\\
Overall it seems that, up to now, there is no general theory on non-stationary continuous-time models available that encompasses non-Gaussianity and provides statistical results.\\
In the present work, we address this problem by investigating stationary approximations in the spirit of \cite{DRW2019} for sequences of non-parametric non-stationary continuous-time processes. Based on such approximations, we show hereditary properties under transformations and establish asymptotic results for localized sample moments of such processes. These asymptotic results form the foundation for many $M$-estimation procedures, to which powerful inference methodologies as the maximum likelihood estimator belong. Such $M$-estimators are investigated in \cite{S2021} building on the results of the present paper.\\ 
All asymptotic results will be obtained by assuming that the stationary approximation is $\theta$-weakly dependent as introduced in \cite{DD2003}. $\theta$-weak dependence is inherited under the same transformations as mentioned above such that our results also apply for a large class of transformations, including the variance and covariance operator.\\
As an example, we embed the results from \cite{BS2021} into our theory. There the authors define a locally stationary process as a sequence of moving average processes $Y_N(t)$ with a time-varying kernel function $g_N(t,t-\cdot)$ driven by a L\'evy process, i.e.
\begin{gather}\label{equation:locstatintroduction}
Y_N(t)=\int_{\R} g_N(t,t-s)L(ds),
\end{gather}
where the kernel function $g_N(Nt,\cdot)$ converges in $L^2$ to a continuous square integrable limiting kernel function $g(t,\cdot)$ as $N$ tends to infinity. Noticeable examples of (\ref{equation:locstatintroduction}) come from the class of time-varying L\'evy-driven state space models, for which the time-varying kernel function, as well as the limiting kernel function, are of exponential type (see Section \ref{sec5-3} for a detailed discussion). In \cite{BS2021} it is shown that, under conditions on the limiting kernel $g_N$ and the L\'evy process $L$, time-varying L\'evy-driven state space models are locally stationary with respect to the given definition. \\
Exemplary, we consider sequences of time-varying L\'evy-driven state space models $Y_N(t)$ in the form of (\ref{equation:locstatintroduction}) and give sufficient conditions for the existence of a stationary approximation $\tilde{Y}_u(t)$ for a fixed approximation point $u$. This approximation can be expressed in terms of the limiting kernel function $g$ from \cite{BS2021} as
\begin{gather}\label{equation:locstatapproxintroduction}
\tilde{Y}_u(t)=\int_\R g(u,t-s)L(ds).
\end{gather} 

The paper is structured as follows. In Section \ref{sec2} we start by introducing locally stationary approximations and discuss the different types of observations which will be considered in this work. Section \ref{sec2-2} is dedicated to different dependence measures, which will be needed to establish asymptotic results. $L^1$-mixingale type dependence measures, $\theta$-weak dependence, and important results on their relationship will be reviewed. Then, hereditary properties of locally stationary approximations and hereditary properties of $\theta$-weak dependence under both finite and infinite transformations are discussed in Section \ref{sec2-3} and \ref{sec2-4}. Based on the locally stationary approximations from Section \ref{sec2-1} we give several asymptotic results in Section \ref{sec3}. Global and local laws of large numbers are proved in Section \ref{sec3-1} under minimal moment assumptions for different observation schemes. In Section \ref{sec3-2} we establish central limit type results for similar observation schemes. The proofs are mainly based on asymptotic theory for stationary triangular arrays \cite{DM2002}. Sufficient conditions from \cite{DM2002} are connected to the $\theta$-weak dependence coefficients of the stationary approximation. In Section \ref{sec4-1} we review elementary properties of L\'evy processes, including stochastic integration with respect to them. A brief summary of the main results from \cite{BS2021} form Section \ref{sec4-2} where we additionally show that sequences of processes in the form (\ref{equation:locstatintroduction}) that possess a locally stationary approximation are locally stationary in the sense of \cite{BS2021}. In Section \ref{sec5} we give, starting from time-varying L\'evy-driven Ornstein-Uhlenbeck processes, a comprehensive introduction on general time-varying L\'evy-driven state space models, time-varying L\'evy-driven CARMA models, and comment on their relationship. Sufficient conditions for the existence of locally stationary approximations for these processes are given in Section \ref{sec5-3}. Finally, we apply the developed asymptotic theory from Section \ref{sec3} and provide asymptotic results for the first and second order localized sample moments of time-varying L\'evy-driven state space models under different observations. Results from \cite{CSS2020} are used that give sufficient conditions for processes in the form (\ref{equation:locstatapproxintroduction}) to be $\theta$-weakly dependent if the kernel is of exponential decay.

\subsection{Notation}
\label{sec1-1}
Throughout this paper, we denote the set of positive integers by $\N$, non-negative integers by $\N_0$, positive real numbers by $\R^+$, non-negative real numbers by $\R^+_0$, the set of $m\times n$ matrices over a ring $R$ by $M_{m\times n}(R)$ and $\mathbf{1}_n$ stands for the $n\times n $ identity matrix. For square matrices $A,B\in M_{n\times n}(R)$, $[A,B]=AB-BA$ denotes the commutator of $A$ and $B$. Norms of matrices and vectors are denoted by $\norm{\cdot}$. If the norm is not further specified, we take the Euclidean norm or its induced operator norm, respectively. For a bounded function $h$, $\norm{h}_\infty$ denotes the uniform norm of $h$. In the following Lipschitz continuous is understood to mean globally Lipschitz. For $u,n\in\N$, let $\mathcal{G}_u^*$ be the class of bounded functions from $(\R^n)^u$ to $\R$ and $\mathcal{G}_u$ be the class of bounded, Lipschitz continuous functions from $(\R^n)^u$ to $\R$ with respect to the distance $\sum_{i=1}^{u}\norm{x_i-y_i}$, where $x,y\in(\R^n)^u$. For $G\in\mathcal{G}_u$ we define
\begin{gather*}
Lip(G)=\sup_{x\neq y}\frac{|G(x)-G(y)|}{\norm{x_1-y_1}+\ldots+\norm{x_u-y_u}}.
\end{gather*} 
We shortly write the transpose of a matrix $A \in M_{m\times n}(\R)$ as $A'$. For two sets $U$ and $V$, $|U|$ stands for the cardinality of a set $U$ and $U\triangle V$ for the symmetric difference of the two sets.
The Borel $\sigma$-algebras are denoted by $\mathcal{B}(\cdot)$ and $\lambda$ stands for the Lebesgue measure, at least in the context of measures. For a normed vector space $W$ we denote by $\ell^\infty(W)$ the space of all bounded sequences in $W$. 
In the following, we will assume all stochastic processes and random variables to be defined on a common complete probability space $(\Omega,\mathcal{F},P)$ equipped with an appropriate filtration if necessary. 
Finally, we simply write $L^p$ to denote the space $L^p(\Omega,\mathcal{F},P)$ and $L^p(\R)$ to denote the space $L^p(\R,\mathcal{B}(\R),\lambda)$ with corresponding norms $\norm{\cdot}_{L^p}$. 

\section{Locally stationary approximations in continuous-time and\\ transformations}
\label{sec2}

\subsection{Locally stationary approximations in continuous-time}
\label{sec2-1}
In this paper, we follow the intuitive idea of local stationarity and assume that a sequence of non-stationary processes $Y_N(t)$ can be locally approximated by a stationary process $\tilde{Y}_u(Nt)$, whenever $t$ is close to $u$ and $N$ increases. We consider this approximation to hold in an $L^p$-sense, i.e. 
\begin{gather*}
\norm{Y_N(t)-\tilde{Y}_u(Nt)}_{L^p}
\end{gather*}
is small for $t$ close to $u$. Considering a rescaled time domain $Nt$ for the stationary process $\tilde{Y}_u(Nt)$, enables us to establish asymptotic results (see the introduction of \cite{D2012} for a comprehensive discussion on the effects of rescaling).\\ 
Understanding this approximation in an $L^p$-sense has several advantages. On the one hand, we show that many interesting asymptotic results carry over from the stationary approximation to the non-stationary process while, on the other hand, many continuous-time models are very well understood in an $L^p$-sense such that we present locally stationary approximations for a wide range of non-stationary models. This idea is motivated by recent work in \cite{DRW2019}, where the authors used similar locally stationary approximations in a discrete-time setting to obtain various analytical and statistical results. \\
Mathematically, we express the idea of a locally stationary approximation as follows.

\begin{Definition}\label{definition:statapproxconttime}
Let $Y_N=\{Y_N(t),t\in\R\}_{N\in\N}$ be a sequence of real-valued stochastic processes and $\tilde{Y}=\{\tilde{Y}_u(t),t\in\R\}_{u\in\R^+}$ a family of real-valued stationary processes. We assume that the process $\tilde{Y}_u$ is ergodic for all $u\in\R^+$ and $\sup_{u\in \R^+} \lVert \tilde{Y}_u(0)\rVert_{L^p}<\infty$ for some $p\geq1$. If there exists a constant $C>0$, such that uniformly in $t\in \R$ and $u,v\in\R^+$
\begin{gather}
\norm{\tilde{Y}_u(t)-\tilde{Y}_v(t)}_{L^p}\leq C \abs{u-v} \text{ and } \norm{Y_N(t)-\tilde{Y}_{t}(Nt)}_{L^p} \leq C \frac{1}{N},\label{assumption:LS}\tag{LS}
\end{gather}
then we call $\tilde{Y}_u$ a \normalfont{locally stationary approximation} \textit{of the sequence} $Y_N$ \normalfont{for} $p$. 
\end{Definition}

Examples of sequences $Y_N$ with locally stationary approximation $\tilde{Y}_u$ are given in Section \ref{sec5}. In particular, we derive sufficient conditions for the existence of locally stationary approximations of time-varying L\'evy-driven state space models. \\

Throughout this paper, we assume that either continuous-time observations of the process $Y_N$ are available or that the process has been sampled in the following way.

\begin{Assumption}\label{assumption:observations}
For fixed $N\in\N$ and $u\in\R^+$ we assume $Y_N$ to be equidistantly observed at times $\tau_i^N=u+i\delta_N$ with grid size $\delta_N=|\tau_i^N-\tau_{i-1}^N|$ such that $\delta_N\downarrow 0$ for $N\rightarrow\infty$. For a sequence $b_N\downarrow0$ we consider the observation window $[u-b_N,u+b_N]$ and set $m_N=\lfloor\frac{b_N}{\delta_N}\rfloor$. Thus, the number of observations is given by $2m_N+1=|\{i\in\Z: \tau_i^N\in [u-b_N,u+b_N]\}|$. We require $\frac{b_N}{\delta_N}\rightarrow\infty$ as $N\rightarrow\infty$ and either
\begin{enumerate}[label={(O\arabic*)}]
\item $N\delta_N=\delta>0$ for all $N\in\N$ or \label{observations:O1}
\item $N\delta_N\rightarrow\infty$ as $N\rightarrow\infty$.\label{observations:O2}
\end{enumerate}
\end{Assumption}

Note that these conditions on $N$, $b_N$ and $\delta_N$ immediately imply $Nb_N\rightarrow\infty$ as $N\rightarrow\infty$. The two conditions \ref{observations:O1} and \ref{observations:O2} describe how the locally stationary approximation $\tilde{Y}_u$ is observed, when the non-stationary process $Y_N$ is sampled. In the case when \ref{observations:O1} holds, the stationary process $\tilde{Y}_u$ is observed equidistantly. Under \ref{observations:O2} the distance of two observations of the process $\tilde{Y}_u$ tends to $\infty$ as $N\rightarrow\infty$.

\subsection{$L^1$-mixingale type dependence measures and $\theta$-weak dependence}
\label{sec2-2}
Considering a set of observations $\mathcal{T}_N=\{\tau_i^N,i=-m_N,\ldots,m_N\}$ as described above we note that this set contracts (with rate $b_N$) towards $u$ as $N\rightarrow\infty$.\\ 
In fact, this contraction allows us to replace the observations $\{Y_N(\tau)\}_{\tau\in\mathcal{T}_N}$ in $L^p$ by the corresponding observations of the triangular array $\{\tilde{Y}_u(N\tau)\}_{\tau\in\mathcal{T}_N}$ at rate $\frac{1}{N}+b_N$ (see e.g. the proof of Theorem \ref{theorem:lawoflargenumbersdiscreteobs}). As usual, to establish asymptotic results for triangular arrays, it is necessary to impose additional conditions on the dependence structure as $N$ increases.\\
By imposing conditions on the dependence structure of only the locally stationary approximation and avoiding assumptions on $Y_N$, we additionally allow for some model misspecification in $Y_N$ like an additional noise component which is decreasing in $N$ and not captured in $Y_N$. \\
All upcoming asymptotic results will be based on the following $L^1$-mixingale type measure of dependence.

\begin{Definition}\label{definition:L1measuresofdependence}
For a $\sigma$-algebra $\mathscr{M}$ and a real-valued integrable random variable $X$ we define
\begin{enumerate}[label={(\alph*)}]
\item $\gamma(\mathscr{M},X)=\norm{E[X|\mathscr{M}]-E[X]}_{1}$ and
\item $\theta(\mathscr{M},X)=\sup_{g\in  \mathscr{L}_1} \norm{E[g(X)|\mathscr{M}]-E[g(X)]}_{1}$,
\end{enumerate}
where $\mathscr{L}_1=\{g:\R\rightarrow \R,$  bounded and Lipschitz with $Lip(g)\leq1\}$.
For a stationary and integrable stochastic process $(X_k)_{k\in\Z}$ we define the following dependence coefficients 
\begin{align*}
\gamma_h&= \gamma\big(\mathscr{M}_0,X_h\big) \text{ and }\\
\theta_h&=\theta\big(\mathscr{M}_0,X_h\big),
\end{align*}
where $h\in\N$ and $\mathscr{M}_0=\sigma(\{X_k,k\leq0\})$.
\end{Definition}

In particular, these measures will be readily accessible for many models through their connection to the following notion of weak dependence, called $\theta$-weak dependence.

\begin{Definition}\label{thetaweaklydependent}
Let $X=\{X(t)\}_{t\in\R}$ be a real-valued stochastic process. Then, $X$ is called $\theta$-weakly dependent if
\begin{gather*}
\theta(h)=\sup_{v\in\N}\theta_{v}(h) \underset{h\ra\infty}{\longrightarrow} 0,
\end{gather*} 
where
\begin{align*}
\theta_{v}(h)\!=\!\sup\bigg\{\frac{|Cov(F(X(i_1),\ldots,X(i_v)),G(X(j)))|}{\norm{F}_{\infty}Lip(G)}, F\in\mathcal{G}^*_v,G\in\mathcal{G}_1, i_1\leq\ldots\leq i_v\leq i_v+h\leq j \bigg\}.
\end{align*}
We call $(\theta(h))_{h\in\R_0^+}$ the $\theta$-coefficients. 
\end{Definition}

The following lemma gives a collection of auxiliary statements that help to connect the $L^1$-mixingale-type measure of dependence of a process with respect to its natural filtration to its $\theta$-weak dependence coefficient. 

\begin{Lemma}\label{lemma:mixingaleandthetaweak}
For a $\sigma$-algebra $\mathscr{M}$, a real-valued integrable random variable $Y$ and a real-valued stochastic process $X=\{X(t)\}_{t\in\R}$ it holds
\begin{enumerate}[label={(\alph*)}]
\item$\gamma(\mathscr{M},Y)\leq\theta(\mathscr{M},Y)\text{ such that }\gamma_h\leq\theta_h, \quad h\in\N$,
\item$\theta_h=\theta(h),  \quad h\in\N,$
\end{enumerate}
where $\theta(h)$ is the $\theta$-coefficient of $X$.\\ 
Now, additionally assume that $X$ is centered and $\theta$-weakly dependent with $\theta$-coefficient $\theta(h)$, $\mathscr{M}_0=\sigma(\{X_k,k\in\Z,k\leq0\})$ and $E[|X_0|^{2+\varepsilon}]<\infty$ for some $\varepsilon>0$. Then,
\begin{enumerate}[label={(\alph*)}]
\setcounter{enumi}{2}
\item $\sum_{k=1}^\infty\norm{X_0E[X_k|\mathscr{M}_0]}_{L^1}\leq D \left(\sum_{h=1}^\infty(h+1)^{\frac{1}{\varepsilon}}\theta(h)\right)^{\frac{\varepsilon}{1+\varepsilon}}$ for some constant $D>0$.
\end{enumerate}
\end{Lemma}
\begin{proof}
The first part of this Lemma is clear. The second part follows from \cite[Lemma 5.1]{CSS2020} for $m=1$. The third part follows from the proof of \cite[Lemma 2 and Corollary 1]{DD2003}. 
\end{proof}

Later, Lemma \ref{lemma:mixingaleandthetaweak} will pave the way to show asymptotic results for sequences of processes that possess a locally stationary approximation that is $\theta$-weakly dependent.

\subsection{Hereditary properties under finite memory transformations}
\label{sec2-3}
In this and the next subsection we investigate hereditary properties of locally stationary approximations and $\theta$-weak dependence under transformations. \\
Consider a sequence of stochastic processes $Y_N$ with locally stationary approximation $\tilde{Y}_u$ for some $p\geq1$. For $k\in\N_0$ define the finite memory vectors
\begin{align*}
Z_N^k(t)=\left(Y_N(t),\ldots,Y_N\left(t-\frac{k}{N}\right)\right) \text{ and } \tilde{Z}_u^k(t)=(\tilde{Y}_u(t),\ldots,\tilde{Y}_u(t-k)). 
\end{align*}

\begin{Definition}[{\cite[Definition 2.4]{DRW2019}}]
A measurable function $g:\R^{k+1}\rightarrow\R$ is said to be in the class $\mathcal{L}_{k+1}(M,C)$ for $M,C\geq0$, if 
\begin{align*}
\sup_{x\neq y}\frac{|g(x)-g(y)|}{\norm{x-y}_1(1+\norm{x}_1^M+\norm{y}_1^M)}\leq C.
\end{align*}
\end{Definition}

\begin{Proposition}\label{proposition:inheritanceproperties}
Let $Y_N$ be a sequence of stochastic processes with locally stationary approximation $\tilde{Y}_u$ for some $\tilde{p}=p(M+1)$, where $p\geq1$, $M\geq0$. Then, for $g\in\mathcal{L}_{k+1}(M,C)$, it holds:
\begin{enumerate}[label={(\alph*)}]
\item $g(\tilde{Z}_u^k)$ is a locally stationary approximation of the sequence $g(Z_N^k)$ for $p$.
\item If $\tilde{Y}_u$ is $\theta$-weakly dependent with $\theta$-coefficients $\theta_{\tilde{Y}_u}(h)$, then $\tilde{Z}_u^k$ is $\theta$-weakly dependent with $\theta$-coefficients $\theta_{\tilde{Z}_u^k}(h)\leq (k+1) \theta_{\tilde{Y}_u}(h-(k+1))$ for $h\geq (k+1)$.
\item If $\tilde{Y}_u$ is $\theta$-weakly dependent with $\theta$-coefficients $\theta_{\tilde{Y}_u}(h)$, $E[|\tilde{Y}_u(t)|^{(1+M+\varepsilon)}]<\infty$ for some $\varepsilon>0$ and additionally $|g(x)|\leq \tilde{C} \norm{x}_1^{M+1}$ for a constant $\tilde{C}>0$, then $g(\tilde{Z}_u^k)$ is $\theta$-weakly dependent with $\theta$-coefficients 
\begin{gather*}
\theta_{g(\tilde{Z}_u^k)}(h)=\mathcal{O}\left(\theta_{\tilde{Y}_u}(h)^{\frac{\varepsilon}{M+\varepsilon}}\right).
\end{gather*}
\end{enumerate}
\end{Proposition}
\begin{proof}
Part (a) follows immediately from Hoelder's inequality. For part (b) we refer to \cite[Lemma 1]{DMT2008}. Finally, part (c) follows from \cite[Lemma 6]{BDL2008}.
\end{proof}

\begin{Remark}\label{remark:inheritancepolynomials}
Proposition \ref{proposition:inheritanceproperties} implies that all results on $Y_N$, that we give in the next sections, are also valid for transformations $g(Y_N^k)$ with $g\in\mathcal{L}_{k+1}(M,C)$ under additional moment conditions. Important examples of functions in $\mathcal{L}_{k+1}(M,C)$ are polynomials $h:\R^{k+1}\rightarrow\R$ with degree at most $(M+1)$, satisfying $h(0)=0$. Note that this also includes the covariance operator at lag $k$, i.e. $g(x_0,x_1,\ldots,x_k)=x_0x_k$, such that $g\in\mathcal{L}_{k+1}(1,1)$.
\end{Remark}

\begin{Remark}
In fact, the condition $|g(x)|\leq C \norm{x}_1^{M+1}$ from Proposition \ref{proposition:inheritanceproperties} (c) can be replaced by the equivalent condition $g(0)=0$ (see \cite[Proposition 3.4]{CS2018}).
\end{Remark}

\subsection{Hereditary properties under infinite memory transformations}
\label{sec2-4}
Although the class $\mathcal{L}_{k+1}$ from above includes important transformations as the covariance operator, it does not cover transformations of vectors that are of infinite memory, i.e. 
\begin{align*}
Z_N(t)=\left(Y_N(t),Y_N\left(t-\frac{1}{N}\right),\ldots\right)\text{ and }\tilde{Z}_u(t)=(\tilde{Y}_u(t),\tilde{Y}_u(t-1),\ldots).
\end{align*} 
Controlling such transformations is for instance important to prove asymptotic properties of certain estimators that are defined in terms of the full history of a process. The following class of functions will preserve locally stationary approximations of such infinite memory vectors and is a modification of equation (9) in \cite{BDW2020}.
%

\begin{Definition}\label{definition:functionclassinfinite}
A measurable function $g:\R^\infty\rightarrow\R$ belongs to the class $\mathcal{L}_\infty^{p,q}(\alpha)$ for $p,q \geq1$, if there exists a sequence $\alpha=(\alpha_k(g))_{k\in\N_0}\subset\R_0^+$ satisfying $\sum_{k=0}^\infty\alpha_k(g)<\infty$ and a function $f:\R_0^+\rightarrow\R_0^+$ such that for all sequences $X=(X_k)_{k\in\N_0}\in \ell^\infty(L^q)$ and $Y=(Y_k)_{k\in\N_0}\in\ell^\infty(L^q)$ it holds
\begin{align*}\label{equation:functionclassinfinite}
\norm{g(X)-g(Y)}_{L^p}\leq f\left(\sup_{k\in\N_0}\{\norm{X_k}_{L^q}\vee \norm{Y_k}_{L^q}\}\right) \sum_{k=0}^\infty \alpha_k(g)\norm{X_k-Y_k}_{L^q}.
\end{align*}
\end{Definition}

The following result follows immediately from the Definition of the class $\mathcal{L}_\infty^{p,q}$.

\begin{Proposition}\label{proposition:inheritancepropertiesinfinite}
Let $Y_N$ be a sequence of stochastic processes with locally stationary approximation $\tilde{Y}_u$ for some $q\geq1$ and $g\in\mathcal{L}_\infty^{p,q}(\alpha)$ for some $p\geq 1$ such that $\sum_{k=0}^\infty k\alpha_k(g)<\infty$. Then, $g(\tilde{Z}_u(t))$ is a locally stationary approximation of $g(Z_N(t))$ for $p$.
\end{Proposition}

As we will see in the next remark, the class $\mathcal{L}_{k+1}$ can be embedded in the class $\mathcal{L}_\infty^{1,(M+1)}$.

\begin{Remark}
Asumme that $g\in \mathcal{L}_{k+1}(M,C)$. Then, $\tilde{g}=\left(g,0,\ldots\right)\in\mathcal{L}_\infty^{1,(M+1)}(\alpha)$, where 
\begin{gather*}
\alpha_n(\tilde{g})=\begin{cases}1,n\leq k+1,\\ 0,n>k+1\end{cases} \text{ and } f(x)=C(1+2(k+1)x^M).
\end{gather*}
\end{Remark}

\section{Asymptotic results}
\label{sec3} 
 
In this section, we discuss various asymptotic results including law of large numbers and central limit type results for sequences of stochastic processes that possess a locally stationary approximation. These asymptotic results provide the foundation for many statistical inference methodologies.

\subsection{Laws of large numbers}
\label{sec3-1}

We give a series of laws of large numbers for different observation schemes. 

\subsubsection{Law of large numbers under continuous observations}
\label{sec3-1-1}
In the following, we give a global and a local law of large numbers. In order to obtain localized estimators, we introduce localizing kernels.

\begin{Definition}\label{definition:localizingkernel}
Let $K:\R\rightarrow \R$ be a bounded function. If $K$ is of bounded variation, has compact support $[-1,1]$ and satisfies $\int_\R K(x)dx=1$, then we call $K$ a localizing kernel.
\end{Definition}

From now on, if not otherwise stated, $K$ always denotes a localizing kernel.


\begin{Theorem}\label{theorem:lawoflargenumberscontobs}
Consider a sequence of stochastic processes $Y_N$ with locally stationary approximation $\tilde{Y}_u$ for some $p\geq1$. Then:
\begin{enumerate}[label={(\alph*)}]
\item For all $t\geq0$,
\begin{gather*}
\frac{1}{t} \int_{0}^t Y_N(\nu)d\nu \overset{L^p}{\underset{N\rightarrow\infty}{\longrightarrow}} \int_{0}^t E[\tilde{Y}_u(0)]du.
\end{gather*}
\item Let $K$ be a localizing kernel that is additionally differentiable and $b_N\downarrow0$ with $Nb_N\rightarrow\infty$. Then, 
\begin{gather*}
\frac{1}{b_N}\int_{u-b_N}^{u+b_N} K\left(\frac{\tau-u}{b_N}\right)Y_N(\tau) d\tau \overset{L^p}{\underset{N\rightarrow\infty}{\longrightarrow}}E[\tilde{Y}_u(0)].
\end{gather*}
\end{enumerate}
\end{Theorem}

%
The following lemma is an ergodic type result for weighted integrals of ergodic processes.

\begin{Lemma}\label{lemma:stationarylocalergodic}
Let $X=\{X(t)\}_{t\in\R}$ be a stationary ergodic process, $\norm{X(0)}_{L^p}<\infty$ for some $p\geq1$ and $K$ a differentiable localizing kernel. Then, for all $u\in\R^+$
\begin{gather*}
\frac{1}{b_N}\int_{u-b_N}^{u+b_N} K\left(\frac{\tau-u}{b_N}\right)X(N\tau) d\tau \overset{L^p}{\underset{N\rightarrow\infty}{\longrightarrow}}E[X(0)],
\end{gather*}
where $b_N\rightarrow 0$ and $Nb_N\rightarrow \infty$ for $N\rightarrow\infty$.
\end{Lemma}
\begin{proof}
First, note that 
\begin{gather*}
\frac{1}{b_N}\int_{u-b_N}^{u+b_N} K\left(\frac{\tau-u}{b_N}\right) d\tau=\int_{-1}^1 K(x)dx=1.
\end{gather*}
Therefore, without loss of generality, we can assume that $Y$ is centered, i.e. $E[Y(0)]=0$.
Let us first assume that $u=0$. Now, using the integration by parts formula
\begin{align*}
\tilde{\mu}_N(0)&\!=\!\frac{1}{b_N}\int_{-b_N}^{b_N} K\left(\frac{\tau}{b_N}\right)X(N\tau) d\tau\\
&\!=\!\frac{1}{b_N}\! \left(\!K(1)\!\!\int_{-b_N}^{b_N}\!\!X(N\tau)d\tau \!-\!K(-1)\!\int_{-b_N}^{-b_N}\!\!X(N\tau)d\tau \!-\!\frac{1}{b_N} \int_{-b_N}^{b_N}\!K'\!\left(\frac{\tau}{b_N}\right) \int_{-b_N}^\tau\! \!\!X(Nu) du \ d\tau\!  \right).
\end{align*}
Since $K$ is of bounded variation, we get
\begin{gather*}
\norm{\tilde{\mu}_N(0)}_{L^p}\leq C_1 \sup_{\tau\in [-b_N,b_N]} \norm{ \frac{1}{b_N} \int_{-b_N}^\tau X(N u) du}_{L^p}=  C_1\sup_{\tau\in [0,2Nb_N]} \norm{ \frac{1}{Nb_N} \int_{0}^\tau X(u) du}_{L^p}
\end{gather*}
for a constant $C_1>0$. To lighten notation, we introduce $F(\tau)=\frac{1}{\tau}\int_0^{\tau}X(u)du$. From the continuous-time version of the Birkhoff-Khinchin theorem \cite[Theorem 9.8]{K1997} we get $F(\tau) \overset{L^p}{\underset{\tau\rightarrow\infty}{\longrightarrow}}E[X(0)]=0$, such that 
\begin{enumerate}[label={(\alph*)}]
\item $\norm{F(\tau)}_{L^p}\leq C_2$ for a constant $C_2>0$ for all $\tau\in\R^+$ and
\item for all $\varepsilon>0$ there exists $c>0$ such that $\norm{F(\tau)}_{L^p}\leq \varepsilon$ for all $\tau>c$. 
\end{enumerate}
Now, let $\frac{\varepsilon}{2C_1}>0$. Consider $c>0$ such that $\norm{F(\tau)}_{L^p}\leq \frac{\varepsilon}{2C_1}$ for all $\tau>c$. Then, for $N$ sufficiently large
\begin{align}
\begin{aligned}
\norm{\tilde{\mu}_N(0)}_{L^p}\leq& 2C_1 \sup_{\tau\in [0,2Nb_N]} \norm{ \frac{1}{2Nb_N} \int_{0}^\tau X(u) du}_{L^p}\\
\leq& 2C_1  \sup_{\tau\in [0,c]} \norm{ \frac{\tau}{2Nb_N} F(\tau)}_{L^p}+2C_1 \sup_{\tau\in [c,2Nb_N]} \norm{ \frac{\tau}{2Nb_N} F(\tau)}_{L^p}\\
\leq& \frac{cC_1}{Nb_N} \sup_{\tau\in [0,c]}\norm{F(\tau)}_{L^p}+ 2C_1\sup_{\tau\in [c,2Nb_N]} \norm{F(\tau)}_{L^p}
\leq\frac{cC_1C_2}{Nb_N}+\varepsilon\underset{N\rightarrow\infty}{\longrightarrow}\varepsilon\label{equation:convergencelemmaloc},
\end{aligned}
\end{align}
which concludes the proof for $u=0$. Since $Y$ is stationary, we obtain for $u\in\R^+$
\begin{align*}
\norm{\tilde{\mu}_N(u)}_{L^p}&=\left\lVert\frac{1}{b_N}\int_{u-b_N}^{u+b_N} K\left(\frac{\tau-u}{b_N}\right)X(N\tau) d\tau\right\rVert_{L^p}\!\! =\norm{\frac{1}{b_N}\int_{-b_N}^{b_N} K\left(\frac{\tau-u}{b_N}\right)X(N\tau+Nu) d\tau}_{L^p}\\
&=\norm{\frac{1}{b_N}\int_{-b_N}^{b_N} K\left(\frac{\tau}{b_N}\right)X(N\tau) d\tau}_{L^p} \underset{N\rightarrow\infty}{\longrightarrow}0
\end{align*}
by (\ref{equation:convergencelemmaloc}).
\end{proof}

\noindent\textit{Proof of Theorem \ref{theorem:lawoflargenumberscontobs}.}
\begin{enumerate}[label={(\alph*)}]
\item For $J,N\in\N$ and $j=1,\ldots,2^J$  we define the intervals 
\begin{gather*}
I_{J,j,N,t}= \left\{s: \frac{s}{N}\in \Big( \frac{t(j-1)}{2^J},\frac{tj}{2^J}\Big] \right\},\text{ such that }\bigcup_{j=1}^{2^J} I_{J,j,N,t} = (0,tN].
\end{gather*}
Notice that $\lambda(I_{J,j,N,t})=\frac{Nt}{2^J}$. By (\ref{assumption:LS}) we have 
\begin{gather*}
\left\lVert \frac{1}{t} \int_{0}^t (Y_N(\nu)-\tilde{Y}_{\nu}(N\nu))d\nu\right\rVert_{L^p}
\leq \sup_{s\in [0,t]} \left\lVert Y_N(s)-\tilde{Y}_{s}(Ns)\right\rVert_{L^p}\underset{N\rightarrow\infty}{\rightarrow}0.
\end{gather*}
Let us fix $J\in\N$. Then,
\begin{align*}
&\norm{\frac{1}{t} \int_{0}^t\tilde{Y}_{\nu}(N\nu)d\nu -\frac{1}{2^J} \sum_{j=1}^{2^J}\frac{1}{\lambda(I_{J,j,N,t})} \int_{I_{J,j,N,t}} \tilde{Y}_{\frac{jt}{2^J}}(u)du}_{L^p} \\
&=\norm{\frac{1}{Nt} \int_{0}^{Nt}\tilde{Y}_{\frac{\nu}{N}}(\nu)d\nu -\frac{1}{2^J} \sum_{j=1}^{2^J}\frac{1}{\lambda(I_{J,j,N,t})} \int_{I_{J,j,N,t}} \tilde{Y}_{\frac{jt}{2^J}}(u)du}_{L^p}\\
&=\norm{\frac{1}{2^J} \sum_{j=1}^{2^J}\frac{1}{\lambda(I_{J,j,N,t})}  \int_{I_{J,j,N,t}} \left( \tilde{Y}_{\frac{\nu}{N}}(\nu) - \tilde{Y}_{\frac{jt}{2^J}}(\nu)\right)d\nu}_{L^p}\\ 
&\leq\sup_{s\in \R_0^+} \sup_{|u-v|\leq\frac{t}{2^J}} \norm{\tilde{Y}_{u}(s) - \tilde{Y}_{v}(s) }_{L^p} \underset{J\rightarrow\infty}{\rightarrow}0
\end{align*}
by (\ref{assumption:LS}). For fixed $J\in\N$ we use the continuous-time version of the Birkhoff-Khinchin theorem \cite[Theorem 9.8]{K1997} to obtain
\begin{gather*}
E(J,N):=\frac{1}{2^J} \sum_{j=1}^{2^J}\frac{1}{\lambda(I_{J,j,N,t})} \int_{I_{J,j,N,t}} \tilde{Y}_{\frac{jt}{2^J}}(u)du \overset{L^p}{\underset{N\rightarrow\infty}{\longrightarrow}} \frac{1}{2^J} \sum_{j=1}^{2^J} E\left[ \tilde{Y}_{\frac{jt}{2^J}}(0) \right]=:E(J).
\end{gather*}
Furthermore, by the continuity of $u\mapsto E[\tilde{Y}_u(0)]$ (follows again from (\ref{assumption:LS})) we have
\begin{gather*}
E(J)= \frac{1}{2^J} \sum_{j=1}^{2^J} E\left[ \tilde{Y}_{\frac{jt}{2^J}}(0) \right]\underset{J\rightarrow\infty}{\longrightarrow}\int_{0}^{t} E[\tilde{Y}_u(0)]du=:E. 
\end{gather*}
Finally,
\begin{align*}
&\norm{\frac{1}{t} \int_{0}^t Y_N(\nu)d\nu -\int_{0}^{t} E[\tilde{Y}_u(0)]du}_{L^p}\\
&\leq\norm{ \frac{1}{t} \int_{0}^t (Y_N(\nu)\!-\!\tilde{Y}_{\nu}(N\nu))d\nu}_{L^p}
\!\!+\!\norm{\frac{1}{t} \int_{0}^t\tilde{Y}_{\nu}(N\nu)d\nu \!-\!\frac{1}{2^J} \sum_{j=1}^{2^J}\frac{1}{\lambda(I_{J,j,N,t})} \int_{I_{J,j,N,t}}\!\! \tilde{Y}_{\frac{jt}{2^J}}(u)du }_{L^p}\\
&\quad+\norm{ E(J,N)-E(J)}_{L^p}+\left|E(J)-E\right|\\
&\leq \sup_{s\in [0,t]} \norm{ Y_N(s)-\tilde{Y}_{s}(Ns)}_{L^p}+ \sup_{s\in [0,Nt]} \sup_{|u-v|\leq\frac{t}{2^J}} \norm{ \tilde{Y}_{u}(s) - \tilde{Y}_{v}(s) }_{L^p} +\norm{ E(J,N)-E(J)}_{L^p}\\&\quad+\left|E(J)-E\right|.
\end{align*}
Then, for all $J\in\N$
\begin{align*}
\limsup_{N\rightarrow\infty}\norm{\frac{1}{t} \int_{0}^t Y_N(\nu)d\nu \!-\!\int_{0}^{t} E[\tilde{Y}_u(0)]du}_{L^p}\!\!
&\leq\sup_{s\in \R_0^+}  \sup_{|u-v|\leq\frac{t}{2^J}} \norm{\tilde{Y}_{u}(s) \!-\! \tilde{Y}_{v}(s)}_{L^p}\!\!+\! \left|E(J)\!-\!E\right| \\
&\leq \sup_{|u-v|\leq\frac{t}{2^J}} C |u-v| + \left|E(J)-E\right|
\end{align*}
by (\ref{assumption:LS}), which converges to $0$ for $J\rightarrow\infty$.
\item It holds
\begin{align*}
&\norm{\frac{1}{b_N}\int_{u-b_N}^{u+b_N} K\left(\frac{\tau-u}{b_N}\right)Y_N(\tau) d\tau -E[\tilde{Y}_u(0)]}_{L^p}\\
&\leq \norm{\frac{1}{b_N}\int_{u-b_N}^{u+b_N} K\left(\frac{\tau-u}{b_N}\right)\left(Y_N(\tau)-\tilde{Y}_\tau(N\tau)\right) d\tau}_{L^p}\\
&\quad+\norm{\frac{1}{b_N}\int_{u-b_N}^{u+b_N} K\left(\frac{\tau-u}{b_N}\right)\left(\tilde{Y}_\tau(N\tau)-\tilde{Y}_u(N\tau)\right) d\tau}_{L^p}\\
&\quad+\norm{\frac{1}{b_N}\int_{u-b_N}^{u+b_N} K\left(\frac{\tau-u}{b_N}\right)\tilde{Y}_u(N\tau) d\tau-E[\tilde{Y}_u(0)]}_{L^p}=:P_1+P_2+P_3.
\end{align*}
Since $P_3$ converges to zero for $N\rightarrow\infty$ by Lemma \ref{lemma:stationarylocalergodic} it remains to analyze $P_1$ and $P_2$. Considering $P_1$ we get
\begin{gather*}
P_1\leq 2 \norm{K}_\infty  \sup_{\tau\in [u-b_N,u+b_N]} \norm{ Y_N(\tau)- \tilde{Y}_\tau(N\tau)}_{L^p}  \underset{N\rightarrow\infty}{\longrightarrow}0
\end{gather*}
by (\ref{assumption:LS}). For the second summand it holds
\begin{align*}
P_2\leq 2 \norm{K}_\infty  \sup_{\tau\in [u-b_N,u+b_N]}\sup_{|u-v|\leq 2b_N} \norm{ \tilde{Y}_u(N\tau)- \tilde{Y}_v(N\tau)}_{L^p}\underset{N\rightarrow\infty}{\longrightarrow}0,
\end{align*}
again by (\ref{assumption:LS}), since $b_N\rightarrow0$.
\end{enumerate}
\qed

\subsubsection{Laws of large numbers under discrete observations}
\label{sec3-1-2}

In this section we give a series of local law of large numbers under discrete observations as considered in Assumption \ref{assumption:observations}. Depending on whether \hyperref[observations:O1]{(O1)} or \hyperref[observations:O2]{(O2)} holds, we obtain different results.\\
We start by assuming that \hyperref[observations:O1]{(O1)} holds and give a discrete-time version of Lemma \ref{lemma:stationarylocalergodic}.

\begin{Lemma}\label{lemma:stationarylocalergodicdiscrete}
Let $X=\{X(t)\}_{t\in\R}$ be a stationary ergodic process and $\norm{X(0)}_{L^p}<\infty$ for some $p\geq1$. Additionally, assume that \ref{observations:O1} holds. Then, for all $u\in\R^+$
\begin{gather*}
\frac{\delta_N}{b_N} \sum_{i=-m_N}^{m_N} K\left(\frac{\tau_i^N-u}{b_N}\right)X(N\tau_i^N) \overset{L^p}{\underset{N\rightarrow\infty}{\longrightarrow}}E[X(0)].
\end{gather*}
\end{Lemma}
\begin{proof}
We first note that, since \ref{observations:O1} holds, we have $N\tau_i^N=Nu+i\delta$ for some $\delta>0$. In principle the proof is a discrete time version of the proof of Lemma \ref{lemma:stationarylocalergodic}. We give the most important steps. Since
\begin{gather*}
\lim_{N\rightarrow\infty }\frac{\delta_N}{b_N} \sum_{i=-m_N}^{m_N} K\left(\frac{\tau_i^N-u}{b_N}\right)=\int_{-1}^1 K(x)dx=1,
\end{gather*}
we can assume without loss of generality that $X$ is centered. For $u=0$ and $S_{k}=\sum_{i=-m_N}^{k}X(i\delta)$ we obtain from the summation by parts formula 
\begin{align*}
\hat{\mu}_N(0)&:=\frac{\delta_N}{b_N} \sum_{i=-m_N}^{m_N} K\left(\frac{\tau_i^N}{b_N}\right)X(i\delta)\\
&=\frac{\delta_N}{b_N}\left(S_{m_N} K\left(\frac{\tau_{m_N}^N}{b_N}\right)+\sum_{k=-m_N}^{m_N-1}S_k \left(K\left(\frac{\tau_k^N}{b_N}\right)-K\left(\frac{\tau_{k+1}^N}{b_N}\right)\right)\right).
\end{align*}
Since $K$ is of bounded variation, we obtain
\begin{align*}
\norm{\hat{\mu}_N(0)}_{L^p}&\leq C_1\frac{\delta_N}{b_N} \sup_{k\in\{-m_N,\ldots,m_N\}} \norm{S_k}_{L^p}\\
&=C_1 \frac{(2m_N+1)\delta_N}{b_N} \sup_{k\in\{0,\ldots,2m_N\}} \norm{\frac{1}{2m_N+1}\sum_{i=0}^{k}X(i\delta)}_{L^p} \underset{N\rightarrow\infty}{\longrightarrow}0
\end{align*}
for a constant $C_1>0$ by similar arguments as in (\ref{equation:convergencelemmaloc}). For general $u\in\R^+$, the result follows from the stationarity of $X$.
\end{proof}

\begin{Theorem}\label{theorem:lawoflargenumbersO1}
Consider a sequence of stochastic processes $Y_N$ with locally stationary approximation $\tilde{Y}_u$ for some $p\geq1$. For observations as given in Assumption \ref{assumption:observations} such that \hyperref[observations:O1]{(O1)} holds, i.e. $N\delta_N=\delta$ for some $\delta>0$, we obtain for all $u\in\R^+$
\begin{gather*}
 \frac{\delta_N}{b_N} \sum_{i=-m_N}^{m_N}K\left(\frac{\tau_i^N-u}{b_N} \right)Y_N(\tau_i^N)\overset{L^p}{\underset{N\rightarrow\infty}{\longrightarrow}}E[\tilde{Y}_u(0)].
\end{gather*}
\end{Theorem}
\begin{proof}
Consider the decomposition
\begin{align*}
&\norm{\frac{\delta_N}{b_N} \sum_{i=-m_N}^{m_N}K\left(\frac{\tau_i^N-u}{b_N} \right)Y_N(\tau_i^N)-E[\tilde{Y}_u(0)]}_{L^p}\\
&\leq \norm{\frac{\delta_N}{b_N} \sum_{i=-m_N}^{m_N}K\left(\frac{\tau_i^N-u}{b_N} \right)Y_N(\tau_i^N)-\frac{\delta_N}{b_N} \sum_{i=-m_N}^{m_N}K\left(\frac{\tau_i^N-u}{b_N} \right)\tilde{Y}_{\tau_i^N}(N\tau_i^N)}_{L^p}\\
&\quad+ \norm{\frac{\delta_N}{b_N} \sum_{i=-m_N}^{m_N}K\left(\frac{\tau_i^N-u}{b_N} \right)\tilde{Y}_{\tau_i^N}(N\tau_i^N)-\frac{\delta_N}{b_N} \sum_{i=-m_N}^{m_N}K\left(\frac{\tau_i^N-u}{b_N} \right)\tilde{Y}_{u}(N\tau_i^N)}_{L^p}\\
&\quad+ \norm{\frac{\delta_N}{b_N} \sum_{i=-m_N}^{m_N}K\left(\frac{\tau_i^N-u}{b_N} \right)\tilde{Y}_{u}(N\tau_i^N)-E[\tilde{Y}_u(0)]}_{L^p}=:P_1+P_2+P_3.
\end{align*}
For $P_1$ and $P_2$ we obtain 
\begin{gather*}
P_1+P_2
\leq \frac{C\norm{K}_\infty\delta_N(2m_N+1)}{b_N} \left(\frac{1}{N}+b_N\right) \underset{N\rightarrow\infty}{\longrightarrow}0
\end{gather*}
for some constant $C>0$. Finally, due to Lemma \ref{lemma:stationarylocalergodicdiscrete}, we also have $P_3\rightarrow0$, as $N\rightarrow\infty$.
\end{proof}

If \ref{observations:O2} holds, we have to impose conditions on the dependence structure of the locally stationary approximation $\tilde{Y}_u$ to obtain similar results as in Theorem \ref{theorem:lawoflargenumbersO1} for $p=1$. Under additional moment conditions a rough inequality extends the result to $p=2$.

\begin{Theorem}\label{theorem:lawoflargenumbersdiscreteobs}
Consider a sequence of stochastic processes $Y_N$ with locally stationary approximation $\tilde{Y}_u$ for some $p\geq1$. For $u\in\R^+$ fixed, assume that $\tilde{Y}_u$ is $\theta$-weakly dependent. Then, for observations as considered in Assumption \ref{assumption:observations} such that \ref{observations:O2} holds, we obtain the following convergences:
\begin{enumerate}[label={(\alph*)}]
\item It holds
\begin{gather*}
\frac{\delta_N}{b_N} \sum_{i=-m_N}^{m_N}K\left(\frac{\tau_i^N-u}{b_N} \right)Y_N(\tau_i^N) \overset{L^1}{\underset{N\rightarrow\infty}{\longrightarrow}}E[\tilde{Y}_u(0)].
\end{gather*}
\item Assume that $\tilde{Y}_u$ is a locally stationary approximation of $Y_N$ for some $p\geq2$ and that $E[|\tilde{Y}_u(0)|^{2+\varepsilon}]<\infty$, for some $\varepsilon>0$, then for all $u\in\R^+$
\begin{gather*}
\frac{\delta_N}{b_N} \sum_{i=-m_N}^{m_N}K\left(\frac{\tau_i^N-u}{b_N} \right)Y_N(\tau_i^N)\overset{L^2}{\underset{N\rightarrow\infty}{\longrightarrow}}E[\tilde{Y}_u(0)].
\end{gather*}
\end{enumerate}
\end{Theorem}

\begin{proof}
\begin{enumerate}[label={(\alph*)}]
\item Consider the same decomposition as in the proof of Theorem \ref{theorem:lawoflargenumbersO1} for $p=1$. In the same way we have $P_1\rightarrow0$ and $P_2\rightarrow0$ as $N\rightarrow\infty$. 
To show that $P_3\underset{N\rightarrow\infty}{\longrightarrow}0$, we apply \cite[Theorem 2]{A1988}. Since
\begin{gather*}
\frac{\delta_N}{b_N} \sum_{i=-m_N}^{m_N} K\left(\frac{\tau_i^N-u}{b_N}\right)\underset{N\rightarrow\infty}{\longrightarrow}\int_{-1}^1 K(x)dx=1,
\end{gather*}
we can assume without loss of generality that $\tilde{Y}_u$ is centered. Consider the centered triangular array $X_{i,2m_N+1}$ and the corresponding filtration $\mathcal{M}_{i,2m_N+1}$, respectively, defined as
\begin{align*}
W_{i,2m_N+1}&=\tilde{Y}_{u}\left(N(u+(i-m_N-1)\delta_N)\right)\\
X_{i,2m_N+1}&=\frac{(2m_N+1)\delta_N}{b_N}K\left(\frac{\tau_{i-m_N-1}^N-u}{b_N} \right)W_{i,2m_N+1}\text{ and}\\
\mathcal{M}_{i,2m_N+1}&:=\sigma\left(\left\{\tilde{Y}_u(N(u+(j-m_N+1)\delta_N)), j\leq i, j\in\Z\right\}\right),\quad i\in\Z,
\end{align*}
where the family $\{\mathcal{M}_{i,2m_N+1}\}_{i=1,\ldots,2m_N+1}$ is non-decreasing.\\
Let us note that the family $\{X_{i,2m_N+1}\}_{ i=1,\ldots,2m_N+1, N\in\N}$ is uniformly integrable due to the stationarity of $\tilde{Y}_u$. Moreover, since $\tilde{Y}_u$ is $\theta$-weakly dependent with $\theta$-coefficients $\theta(h)$, also the sampled process $\{W_{i,2m_N+1}\!\}_{i\in\Z}$ is $\theta$-weakly dependent with $\theta$-coefficients $\theta^{(2m_N+1)}(h)\!=\!\theta(N\delta_Nh)$. Now, because $N\delta_N\rightarrow\infty$ as $N\rightarrow\infty$ and $\theta(h)$ is non-increasing, we have
\begin{gather}\label{equation:suptheta}
\max_{N\in\N}\theta^{(2m_N+1)}(h)= \theta(C_1h)
\end{gather} 
for the constant $C_1=\min_{N\in\N}N\delta_N>0$. In view of Definition \ref{definition:L1measuresofdependence}
and Lemma \ref{lemma:mixingaleandthetaweak} we obtain from the stationarity of $\tilde{Y}_u$ for all $m\in\N_0$
\begin{align*}
&\norm{E[X_{i,2m_N+1}|\mathcal{M}_{(i-m),2m_N+1}]}_{L^1}=\gamma(\mathcal{M}_{(i-m),2m_N+1},X_{i,2m_N+1})\\
&\leq C_2\gamma(\mathcal{M}_{0,2m_N+1},W_{m,2m_N+1}) \leq C_2\theta(\mathcal{M}_{0,2m_N+1},W_{m,2m_N+1})=C_2\theta^{2m_N+1}(m)\\
&\leq C_2\theta(N\delta_Nm)\leq C_2\theta(C_1m)\underset{m\rightarrow\infty}{\longrightarrow}0
\end{align*}
for some constant $C_2>0$, where we used (\ref{equation:suptheta}). 
Overall, we have that the triangular array $X_{i,2m_N+1}$ is an $L^1$-mixingale with respect to the filtration $\mathcal{M}_{i,2m_N+1}$ in the sense of \cite[Definition 2]{A1988}. An application of \cite[Theorem 2]{A1988} gives $P_3\underset{N\rightarrow\infty}{\longrightarrow}0$.
\item Consider the same decomposition as in the proof of Theorem \ref{theorem:lawoflargenumbersO1} for $p=2$. In the same way we have $P_1\rightarrow0$ and $P_2\rightarrow0$ as $N\rightarrow\infty$. Again, without loss of generality we assume that $\tilde{Y}_u$ is centered. For $P_3$ we obtain
\begin{align*}
&\Big\lVert \frac{\delta_N}{b_N} \sum_{i=-m_N}^{m_N}K\left(\frac{\tau_i^N-u}{b_N} \right)\tilde{Y}_{u}(N\tau_i^N)\Big\rVert^2
=\frac{\delta_N^2}{b_N^2}\bigg(\sum_{i=-m_N}^{m_N}Var\left(K\left(\frac{\tau_i^N-u}{b_N} \right)\tilde{Y}_{u}(N\tau_i^N)\right)\\
&\quad+\sum_{i\neq j}Cov\bigg(K\bigg(\frac{\tau_i^N-u}{b_N} \bigg)\tilde{Y}_{u}(N\tau_i^N),K\bigg(\frac{\tau_j^N-u}{b_N} \bigg)\tilde{Y}_{u}(N\tau_j^N)\bigg) \bigg)\\
&\leq \frac{\norm{K}_\infty^2\delta_N^2 (2m_N\!+\!1)^2}{b_N^2}\bigg(\!\frac{E[\tilde{Y}_u(0)^2]}{2m_N+1}\!+\!\sup_{i\neq j} \abs{Cov\!\left(\!\tilde{Y}_{u}(0),\tilde{Y}_{u}(N\delta_N|j\!-\!i|)\!\right)}\bigg)\\
&\leq\frac{\norm{K}_\infty^2\delta_N^2 (2m_N\!+\!1)^2}{b_N^2}\bigg(\!\frac{E[\tilde{Y}_u(0)^2]}{2m_N+1}\!+\!
9E[|\tilde{Y}_{u}(0)|^{2+\varepsilon}]^{\frac{1}{1+\varepsilon}}\theta(N\delta_N)^{\frac{\varepsilon}{1+\varepsilon}} \bigg),
\end{align*}
which converges to zero as $N\rightarrow\infty$, since $\theta(h)\downarrow0$ as $h\rightarrow\infty$ and $i\neq j$. The last inequality follows from \cite[Remark 3.3]{CS2018}. For the convergence we use the results on the $\theta$-coefficients of part (a).
\end{enumerate}
\end{proof}

\subsection{Central limit type results}
\label{sec3-2}
The following theorem gives a central limit type result for sequences of non-stationary processes that possess a locally stationary approximation in the sense of Definition \ref{definition:statapproxconttime}. We again consider the previous observation schemes and restrict ourselves to the rectangular kernel 
\begin{gather}\label{equation:rectangularkernel}
K_{rect}(x)=\frac{1}{2}\mathbb{1}_{\{x\in[-1,1]\}}
\end{gather}
as localization kernel. It is easy to see that $K_{rect}$ is a localizing kernel as defined in Definition \ref{definition:localizingkernel}. Depending on whether \hyperref[observations:O1]{(O1)} or \hyperref[observations:O2]{(O2)} holds, we obtain different asymptotic variances.\\
In addition, we assume the $\theta$-coefficients $\theta(h)$ of the locally stationary approximation to satisfy for a fixed $\varepsilon>0$ the condition
\begin{align*}
\hypertarget{DD}{DD(\varepsilon):} \qquad\sum_{h=1}^\infty\theta(h)h^{\frac{1}{\varepsilon}}<\infty.
\end{align*}
Sufficient conditions for \hyperlink{DD}{DD($\varepsilon$)} to hold are for instance $\theta(h)\in\mathcal{O}(h^{-\alpha})$ or $\theta(h)\in\mathcal{O}(h\ln(h)^{-\alpha})$ for some $\alpha>(1+\frac{1}{\varepsilon})$.

\begin{Theorem}\label{theorem:centrallimit}
Consider a sequence of stochastic processes $Y_N$ with centered locally stationary approximation $\tilde{Y}_u$ for some $p\geq2$ such that $\tilde{Y}_u(0)\in L^{2+\varepsilon}$ for some $\varepsilon>0$. Additionally, assume that $\tilde{Y}_u$ is $\theta$-weakly dependent with $\theta$-coefficients $\theta(h)$ satisfying \hyperlink{DD}{DD($\varepsilon$)}.
Assume that we observe $Y_N$ as described in Assumption \ref{assumption:observations}, such that additionally $\sqrt{m_N}b_N\rightarrow0$. Now, if either 
\begin{enumerate}[label={(\alph*)}]
\item \hyperref[observations:O1]{(O1)} holds, i.e. $N\delta_N=\delta>0$ and $\sigma(u)^2=\frac{1}{2}E[\tilde{Y}_u(0)^2]+\sum_{k=1}^\infty E[\tilde{Y}_u(0)\tilde{Y}_u(k\delta)]>0$ or
\item \hyperref[observations:O2]{(O2)} holds and $\sigma(u)^2=\frac{1}{2}E[\tilde{Y}_u(0)^2]>0$,
\end{enumerate}
then $\sigma(u)^2<\infty$ and 
\begin{align}\label{equation:centrallimitpartialsum}
\sqrt{\frac{\delta_N}{b_N}} \sum_{i=-m_N}^{m_N}K_{rect}\left(\frac{\tau_i^N-u}{b_N}\right)Y_N(\tau_i^N)\overset{d}{\underset{N\rightarrow\infty}{\longrightarrow}} \sigma(u) Z,
\end{align}
where $Z$ is a standard normally distributed random variable and $K_{rect}$ the rectangular kernel as defined in (\ref{equation:rectangularkernel}).
\end{Theorem} 

It is natural to ask for conditions on stronger concepts of convergence as the stated convergence in distribution in (\ref{equation:centrallimitpartialsum}). The following corollary provides additional assumptions that strengthen this convergence to stable convergence. Recall that a sequence of integrable random variables $(Y_n)_{n\in\N}$ converges stably with limit $Y$, where $Y$ is defined on an extension $(\Omega',\mathcal{F}',P')$, if $E[g(Y_n)Z]\longrightarrow E'[g(Y)Z]$ as $n\rightarrow\infty$ for all bounded, continuous functions $g$ and any bounded $\mathcal{F}$-measurable random variable $Z$. Then, we write $Y_n\overset{st}{\underset{n\rightarrow\infty}{\longrightarrow}}Y$. Note that stable convergence immediately implies convergence in distribution.

\begin{Corollary}
Consider a sequence of stochastic processes $Y_N$ with centered locally stationary approximation $\tilde{Y}_u$ for some $p\geq2$ satisfying the assumptions of Theorem \ref{theorem:centrallimit}. If the family
\begin{align*}
\{\mathcal{M}_{i,2m_N+1}\}_{N\in\N}=\sigma\left(\left\{\tilde{Y}_u(N(u+(j-m_N+1)\delta_N)), j\leq i, j\in\Z\right\}\right),\quad i\in\Z
\end{align*}
is non-decreasing in $N$ for all $i\in\Z$, then 
\begin{align*}
\sqrt{\frac{\delta_N}{b_N}} \sum_{i=-m_N}^{m_N}K_{rect}\left(\frac{\tau_i^N-u}{b_N}\right)Y_N(\tau_i^N)\overset{st}{\underset{N\rightarrow\infty}{\longrightarrow}} \sigma(u) Z.
\end{align*}
\end{Corollary}
\begin{proof}
It is sufficient to follow the proof of Theorem \ref{theorem:centrallimit} and use \cite[Theorem 3.7 (i)]{HL2015} as well as \cite[Corollary 1]{DM2002}.
\end{proof}

\noindent In the following lemma we give sufficient conditions for $\{\mathcal{M}_{i,2m_N+1}\}_{N\in\N}$ to be non-decreasing.

\begin{Lemma}\label{lemma:nondecreasingfiltration}
Consider a sequence of stochastic processes $Y_N$ with centered locally stationary approximation $\tilde{Y}_u$ for some $p\geq2$ satisfying the assumptions of Theorem \ref{theorem:centrallimit} such that \hyperref[observations:O1]{(O1)} holds with $N\delta_N=\Delta>0$. Additionally assume that $u$ is a multiple of $\Delta$ and that the increments of $Nb_N$ are bounded by $u$, i.e. $(N+1)b_{N+1}-Nb_N\leq u$. Then, the family $\{\mathcal{M}_{i,2m_N+1}\}_{N\in\N}$ is non-decreasing for all $i\in\Z$.
\end{Lemma}
\begin{proof}
We need to show that for any $N\in\N$, $i\in\Z$ and $j\leq i$ there exists $j_1\in\Z$, $j_1\leq i$ such that $Nu+\delta(j-\lfloor \tfrac{Nb_N}{\delta}\rfloor+1)=(N+1)u+\delta(j_1-\lfloor \tfrac{(N+1)b_{N+1}}{\delta}\rfloor+1)$. We set $j_1=j-\lfloor \tfrac{Nb_N}{\delta}\rfloor+\lfloor \tfrac{(N+1)b_{N+1}}{\delta}\rfloor-\tfrac{u}{\delta}\in\Z$. Then, $j_1\leq i$, since
\begin{align*}
\lfloor \tfrac{(N+1)b_{N+1}}{\delta}\rfloor-\lfloor \tfrac{Nb_N}{\delta}\rfloor\leq \lfloor \tfrac{(N+1)b_{N+1}-Nb_N}{\delta}\rfloor \leq \lfloor \tfrac{u}{\delta}\rfloor =\tfrac{u}{\delta}.
\end{align*}
\end{proof}

\noindent A particular choice for $b_N$ and $\delta_N$ satisfying the assumptions of Lemma \ref{lemma:nondecreasingfiltration} is for instance $b_N=b N^{-\frac{2}{3}}$ and $\delta_N=\Delta N^{-1}$, where $0<b<u$ and $u=m\Delta$ for some $m\in\N$, $\Delta\in\R^+$.\\

\noindent\textit{Proof of Theorem \ref{theorem:centrallimit}.}
Since we obtain from (\ref{assumption:LS})
\begin{align*}
&\bigg\lVert\sqrt{\frac{\delta_N}{b_N}}\!\!\sum_{i=-m_N}^{m_N}\!\!\!\!\!K_{rect}\Big(\frac{\tau_i^N-u}{b_N}\Big)\!\left(\!Y_N(\tau_i^N)\!-\!\tilde{Y}_{\tau_i^N}(N\tau_i^N)\!+\!\tilde{Y}_{\tau_i^N}(N\tau_i^N)\!-\!\tilde{Y}_u(N\tau_i^N)\!\right)\!\bigg\rVert_{L^1}\\
&\leq \frac{ C (2m_N+1)}{2}\sqrt{\frac{\delta_N}{b_N}}\Big(\frac{1}{N}+b_N\Big) \underset{N\rightarrow\infty}{\longrightarrow}0,
\end{align*}
it is enough to show that
\begin{align*}
\sqrt{\frac{\delta_N}{b_N}}\sum_{i=-m_N}^{m_N}K_{rect}\left(\frac{\tau_i^N-u}{b_N}\right)\tilde{Y}_{u}(N\tau_i^N)\overset{d}{\underset{N\rightarrow\infty}{\longrightarrow}} \sigma(u) Z,
\end{align*}
where $Z$ is a standard normally distributed random variable. 
We define for $i\in\Z$ and $N\in\N$
\begin{align*}
W_{i,2m_N+1}&:=\tilde{Y}_u(N\tau_{i-m_N-1}^N)=\tilde{Y}_u(N(u+(i-m_N-1)\delta_N)),\notag\\
X_{i,2m_N+1}&:=\sqrt{2m_N+1}\sqrt{\frac{\delta_N}{b_N}}K_{rect}\left(\frac{\tau_{i-m_N-1}^N-u}{b_N}\right)W_{i,2m_N+1},\notag \\
S_N(t)&:=\sum_{i=1}^{\lfloor t(2m_N+1)\rfloor }X_{i,2m_N+1},\quad S_N:=S_N(1)\quad \text{ and }\notag\\
\mathcal{M}_{i,2m_N+1}&:=\sigma\left(\left\{\tilde{Y}_u(N(u+(j-m_N+1)\delta_N)), j\leq i, j\in\Z\right\}\right).
\end{align*}
As a major first step we show that the conditional central limit theorem \cite[Theorem 2]{DM2002} holds. From this conditional central limit theorem we then derive the convergence in (\ref{equation:centrallimitpartialsum}).\\
Note that $\{\mathcal{M}_{i,2m_N+1}\}_{i\in\N}$ is non-decreasing and 
\begin{gather*}
P\left(\left|\frac{1}{\sqrt{2m_N+1}}X_{0,2m_N+1}\right|>\varepsilon\right)\leq \frac{\norm{\tilde{Y}_u(0)}_{L^1}}{2\varepsilon }\sqrt{\frac{\delta_N}{b_n}}\underset{N\rightarrow\infty}{\longrightarrow}0
\end{gather*}
for all $\varepsilon>0$, since $\tilde{Y}_u$ is stationary. We proceed by considering the conditions (a) and (b) separately.
\begin{enumerate}[label={(\alph*)}]
\item[(b)] Let us assume that \hyperref[observations:O2]{(O2)} holds, i.e. $N\delta_N\rightarrow\infty$ as $N\rightarrow\infty$.
Following \cite[Corollary 3 and Remark 9]{DM2002}, it is enough to show that for all $t\in (0,1]$
\begin{align}
\lim_{N\rightarrow\infty}\sum_{k=1}^{2m_N+1} \norm{X_{0,2m_N+1}E\left[ X_{k,2m_N+1}|\mathcal{M}_{0,2m_N+1}\right]}_{L^1}=&0,\label{equation:CLTproofcond1}\\
\lim_{N\rightarrow\infty}\frac{1}{\sqrt{2m_N+1}}\norm{E\left[S_N(t)\Big|\mathcal{M}_{0,2m_N+1}\right]}_{L^1}=&0 \text{ and }\label{equation:CLTproofcond2}\\
\frac{1}{t(2m_N+1)}\sum_{i=1}^{\lfloor t(2m_N+1)\rfloor } X_{i,2m_N+1}^2 \overset{L^1}{\underset{N\rightarrow\infty}{\longrightarrow}}&\sigma(u)^2,\label{equation:CLTproofcond3}
\end{align}
where $\sigma(u)^2=\frac{1}{2}E[\tilde{Y}_u(0)^2]$ is non-negative.\\
With the help of Lemma \ref{lemma:mixingaleandthetaweak} we will now connect the conditions (\ref{equation:CLTproofcond1}) and (\ref{equation:CLTproofcond2}) to the $\theta$-coefficients of $\tilde{Y}_u$. We start by proving (\ref{equation:CLTproofcond1}). The sampled process $\{W_{i,2m_N+1}\}_{i\in\N}$ is $\theta$-weakly dependent with $\theta$-coefficients $\theta^{(2m_N+1)}(h)\!=\!\theta(N\delta_Nh)$, where $\max_{N\in\N}\!\theta^{(2m_N+1)}(h)\!=\! \theta(C_1h)$ for a constant $C_1=\min_{N\in\N}N\delta_N>0$.
Then, since $\lim_{N\rightarrow\infty}(2m_N+1)\frac{\delta_N}{b_N}=2$, (\ref{equation:CLTproofcond1}) is bounded from above by
\begin{gather*}
\lim_{N\rightarrow\infty}\frac{1}{2}\sum_{k=1}^\infty\norm{W_{0,2m_N+1} E\left[ W_{k,2m_N+1}|\mathcal{M}_{0,2m_N+1}\right]}_{L^1}.
\end{gather*}
For $U_k^{(2m_N+1)}=\norm{W_{0,2m_N+1} E\left[ W_{k,2m_N+1}|\mathcal{M}_{0,2m_N+1}\right]}_{L^1}$, we obtain due to part (c) of Lemma \ref{lemma:mixingaleandthetaweak} and (\ref{equation:suptheta}) 
\begin{align*}
\sum_{k=1}^\infty U_k^{(2m_N+1)}&\leq D \left(\sum_{h=1}^\infty(h+1)^{\frac{1}{\varepsilon}}\theta^{(2m_N+1)}(h)\right)^{\frac{\varepsilon}{1+\varepsilon}}
\leq D \left(\sum_{h=1}^\infty(h+1)^{\frac{1}{\varepsilon}}\theta(C_1h)\right)^{\frac{\varepsilon}{1+\varepsilon}}.
\end{align*}
Note that $(x+y)^p\leq \max(1,2^{p-1})(x^p+y^p)$ for all $x,y,p\in\R_0^+$. Hence, for some $\alpha$ such that $0<\alpha<\min(1,C_1)$ we obtain
\begin{align}
\begin{aligned}\label{equation:seriesthetashifted}
D \left(\sum_{h=1}^\infty(h+1)^{\frac{1}{\varepsilon}}\theta(C_1h)\right)^{\frac{\varepsilon}{1+\varepsilon}}\!\!\!\!
&\leq D \left(\sum_{h=1}^\infty(h+1)^{\frac{1}{\varepsilon}}\theta(\left\lfloor \alpha h\right\rfloor)\right)^{\frac{\varepsilon}{1+\varepsilon}}\\
&= D \left(\sum_{k=0}^\infty \theta(k) \sum_{\{h\in\N:\frac{k}{\alpha}\leq h < \frac{k+1}{\alpha} \}}(h+1)^{\frac{1}{\varepsilon}}\right)^{\frac{\varepsilon}{1+\varepsilon}}\\
&\leq D \left( \left(1+\frac{1}{\alpha}\right) \sum_{k=0}^\infty \theta(k) (\tfrac{k}{\alpha}+\tfrac{1+\alpha}{\alpha})^{\frac{1}{\varepsilon}}\right)^{\frac{\varepsilon}{1+\varepsilon}}\\
&\leq D \left(\left(\max(1,2^{\frac{1}{\varepsilon}-1})\right)  \left(\frac{1+\alpha}{\alpha^{1+\frac{1}{\varepsilon}}}\right)\sum_{k=0}^\infty \theta(k) \left(2^{\frac{1}{\varepsilon}} +k^{\frac{1}{\varepsilon}}\right)  \right)^{\frac{\varepsilon}{1+\varepsilon}},
\end{aligned}
\end{align}
which is finite, since \hyperlink{DD}{DD($\varepsilon$)} holds. By the dominated convergence theorem we obtain 
\begin{gather*}
\lim_{N\rightarrow\infty}D \left(\sum_{h=1}^\infty(h+1)^{\frac{1}{\varepsilon}}\theta(N\delta_Nh)\right)^{\frac{\varepsilon}{1+\varepsilon}}=0, \text{ such that (\ref{equation:CLTproofcond1}) holds.}
\end{gather*}
To show (\ref{equation:CLTproofcond2}), we first note that since $\tilde{Y}_u$ is centered, we have $\norm{E\left[W_{i,2m_N+1}|\mathcal{M}_{0,2m_N+1}\right]}_{L^1}\!=\!\gamma(\mathcal{M}_{0,2m_N+1},W_{i,2m_N+1})\leq \theta(\mathcal{M}_{0,2m_N+1},W_{i,2m_N+1})=\theta^{(2m_N+1)}(i)$, where we used part (a) and (b) of Lemma \ref{lemma:mixingaleandthetaweak}. Thus,
\begin{align*}
&\lim_{N\rightarrow\infty}\frac{1}{\sqrt{2m_N+1}}\!\norm{E\left[S_N(t)\Big|\mathcal{M}_{0,2m_N+1}\right]}_{L^1} \!
\leq\! \frac{1}{\sqrt{2m_N+1}}\!\!\sum_{i=1}^{2m_N+1}\!\!\!\norm{E\left[X_{i,2m_N+1}|\mathcal{M}_{0,2m_N+1}\right]}_{L^1}\\
&\leq\frac{1}{2}\sqrt{\frac{\delta_N}{b_N}}\sum_{i=1}^{2m_N+1}\theta^{(2m_N+1)}(i)=\frac{1}{2}\sqrt{\frac{\delta_N}{b_N}}\sum_{i=1}^{2m_N+1}\theta(N\delta_Ni)
\leq C_2 \sqrt{\frac{\delta_N}{b_N}}\rightarrow 0
\end{align*}
as $N\rightarrow\infty$ for some constant $C_2>0$, since \hyperlink{DD}{DD($\varepsilon$)} holds, using similar arguments as in (\ref{equation:suptheta}) and (\ref{equation:seriesthetashifted}).\\
To show (\ref{equation:CLTproofcond3}) for $\sigma(u)^2=\frac{1}{2}E[\tilde{Y}_u(0)^2]$, we define for $t\in(0,1]$ 
\begin{gather}\label{equation:localizingkernelgtx}
G_t(x)=\frac{2K_{rect}^2(x)}{t}\mathbb{1}_{\left\{x\in [-1,-1+2t]\right\}},\ x\in\R.
\end{gather} 
It is easy to see that $G_t$ is a localizing kernel for all $t\in(0,1]$. As the function $x\mapsto x^2\in\mathcal{L}_1(1,1)$, Proposition \ref{proposition:inheritanceproperties} ensures that the stationary process $\tilde{Y}_u^2$ is $\theta$-weakly dependent. Then, following the same steps as in the proof of Theorem \ref{theorem:lawoflargenumbersdiscreteobs} part $(a)$, using the localizing kernel $G_t(x)$, we obtain
\begin{gather*}
\norm{\frac{\delta_N}{b_N} \sum_{i=-m_N}^{m_N}\frac{1}{2}G_t\left(\frac{\tau_i^N-u}{b_N} \right)\tilde{Y}_{u}(N\tau_i^N)^2-\frac{E[\tilde{Y}_u(0)^2]}{2}}_{L^1}\underset{N\rightarrow\infty}{\longrightarrow}0.
\end{gather*}
Therefore, to show (\ref{equation:CLTproofcond3}) for $\sigma(u)^2=\frac{1}{2}E[\tilde{Y}_u(0)^2]$, it is enough to observe
\begin{align}
\begin{aligned}\label{eq:asymptoticvarianceCLTO2}
&\frac{\delta_N}{b_N}\Bigg\lVert\frac{1}{t} \sum_{i=1}^{\lfloor t(2m_N+1)\rfloor } K_{rect}\left(\frac{\tau_{i-m_N-1}^N-u}{b_N} \right)^2\tilde{Y}_{u}(N\tau_{i-m_N-1}^N)^2\\
&\qquad - \sum_{i=-m_N}^{m_N}\frac{1}{2}G_t\left(\frac{\tau_i^N-u}{b_N} \right)\tilde{Y}_{u}(N\tau_i^N)^2\Bigg\rVert_{L^1}\\
&\quad \leq\frac{E[\tilde{Y}_u(0)^2]}{4t}\frac{\delta_N}{b_N}  \sum_{i=1}^{2m_N+1 } \abs{ \mathbb{1}_{\left\{i\in[1,\lfloor t(2m_N+1)\rfloor ]\right\}} -\mathbb{1}_{\left\{\frac{(i-m_N-1)\delta_N}{b_N}\in[-1,2t-1]\right\}}}\\
&\quad =\frac{E[\tilde{Y}_u(0)^2]}{4t}\frac{\delta_N}{b_N}  \sum_{i=1}^{2m_N+1 } \abs{ \mathbb{1}_{\left\{i\in[1,\lfloor t(2m_N+1)\rfloor ]\right\}} -\mathbb{1}_{\left\{i\in\left[1,2t\frac{b_N}{\delta_N}-\frac{b_N}{\delta_N}+m_N+1\right]\right\}}}\\
&\quad\leq\frac{E[\tilde{Y}_u(0)^2]}{4t}\frac{\delta_N}{b_N} \Bigg|\left\{i\in\Z, i\in[1,\lfloor t(2m_N+1)\rfloor ]\triangle\left[1,2t\frac{b_N}{\delta_N}-\frac{b_N}{\delta_N}+m_N+1\right]  \right\}\Bigg|\\
&\quad \leq \frac{E[\tilde{Y}_u(0)^2]}{4t}\frac{\delta_N}{b_N} \left(\abs{\lfloor t(2m_N+1)\rfloor-2t\frac{b_N}{\delta_N}+\frac{b_N}{\delta_N}-m_N-1}+1\right)\underset{N\rightarrow\infty}{\longrightarrow}0,
\end{aligned}
\end{align}
since $\left(\abs{\lfloor t(2m_N+1)\rfloor-2t\frac{b_N}{\delta_N}+\frac{b_N}{\delta_N}-m_N-1}+1\right)$ is bounded by
\begin{align*}
\abs{\lfloor t(2m_N+1)\rfloor-2t\frac{b_N}{\delta_N}-1}+\abs{m_N-\frac{b_N}{\delta_N}}+1\leq2t +2\leq 4.
\end{align*}
Using \cite[Corollary 3 and Remark 9]{DM2002}, the conditional central limit theorem \cite[Theorem 2]{DM2002} holds and for $t=1$ we have
\begin{gather}\label{equation:condCLT}	
\lim_{N\rightarrow\infty}\norm{E\left[\varphi\left(\frac{1}{\sqrt{2m_N+1}}S_N\right)-\int_\R\varphi(x\sigma(u))g(x)dx\Bigg|\mathcal{M}_{k,2m_N+1} \right]}_{L^1}=0
\end{gather}
for all $k\!\in\!\N$ and $\varphi\! \in\!\mathcal{H}\!=\!\{\varphi\!:\!\R\rightarrow\R,\text{continuous such that }x\!\mapsto \!|(1+x^2)^{-1}\varphi(x)| \text{ is bounded}\}$. In order to show the stated convergence in (\ref{equation:centrallimitpartialsum}), it is sufficient to observe that for all continuous and bounded functions $h$ and a standard normally distributed random variable $Z$ it holds
\begin{align}
\begin{aligned}\label{equation:fromL1toconvergenceindistribution}
&\abs{E\left[h\left(\frac{1}{\sqrt{2m_N+1}}S_N\right)\right]-E\left[h\left(\sigma(u) Z\right)\right]}\\&\leq\norm{E\left[h\left(\frac{1}{\sqrt{2m_N+1}}S_N\right)-\int_\R h\left(x\sigma(u)\right)g(x)dx\Bigg|\mathcal{M}_{1,2m_N+1} \right]}_{L^1}\underset{N\rightarrow\infty}{\longrightarrow}0,
\end{aligned}
\end{align}
since $h\in\mathcal{H}$ and (\ref{equation:condCLT}) holds.
\item[(a)] Let us assume that \hyperref[observations:O1]{(O1)} holds, i.e. $N\delta_N=\delta>0$. In the following we use \cite[Corollary 3]{DM2002}, where we define
\begin{gather*}
\sigma(u)^2=\frac{1}{2}E[\tilde{Y}_u(0)^2]+\sum_{k=1}^\infty E[\tilde{Y}_u(0)\tilde{Y}_u(k\delta)]<\infty.
\end{gather*}
Let us briefly mention that, due to the decay of $\theta(h)$ (i.e. since \hyperlink{DD}{DD($\varepsilon$)} holds) and \cite[Remark 3.3]{CS2018}, $\sigma(u)^2=\frac{E[\tilde{Y}_u(0)^2]}{2}+\sum_{k=1}^\infty E[\tilde{Y}_u(0)\tilde{Y}_u(k\delta)]$ is finite.\\
Overall, we obtain a similar convergence as given in (\ref{equation:condCLT}), if for any $t\in(0,1]$ and $k\in\N_0$
\begin{align}
\lim_{K\rightarrow\infty}\limsup_{N\rightarrow\infty }\sup_{K\leq m\leq (2m_N+1)}\norm{X_{0,2m_N+1}\sum_{k=K}^mE[X_{k,2m_N+1}|\mathcal{M}_{0,2m_N+1}]}_{L^1}&=0,\label{equation:CLTproofcondO11}\\
\lim_{N\rightarrow\infty}\frac{1}{\sqrt{2m_N+1}}\norm{E\left[S_N(t)\Big|\mathcal{M}_{0,2m_N+1}\right]}_{L^1}&=0\text{ and}\label{equation:CLTproofcondO12}\\
\frac{1}{t(2m_N+1)}\sum_{i=1}^{\lfloor t(2m_N+1)\rfloor } X_{i,2m_N+1}X_{i+k,2m_N+1} \overset{L^1}{\underset{N\rightarrow\infty}{\longrightarrow}}&\lambda_k,\label{equation:CLTproofcondO13}
\end{align}
where $\lambda_k=E[\tilde{Y}_u(0)\tilde{Y}_u(k\delta)]$. We first note that (\ref{equation:CLTproofcondO12}) follows analogously to the proof of (\ref{equation:CLTproofcond2}).\\
To show (\ref{equation:CLTproofcondO11}), we follow similar steps as in the proof of (\ref{equation:CLTproofcond1}) and obtain for $K\in\N$
\begin{align*}
&\limsup_{N\rightarrow\infty }\sup_{K\leq m\leq (2m_N+1)}\norm{X_{0,2m_N+1}\sum_{k=K}^mE[X_{k,2m_N+1}|\mathcal{M}_{0,2m_N+1}]}_{L^1}\\
&\quad \leq \frac{1}{2}\limsup_{N\rightarrow\infty } \sum_{k=K}^\infty \norm{W_{0,2m_N+1}E[W_{k,2m_N+1}|\mathcal{M}_{0,2m_N+1}]}_{L^1}\\
&\quad = \frac{1}{2}\limsup_{N\rightarrow\infty } \sum_{k=K}^\infty \norm{\tilde{Y}_u(0)E[\tilde{Y}_u(k\delta)|\sigma(\{\tilde{Y}_u(j\delta):j\leq0,j\in\Z\})]}_{L^1}
=:V(K),
\end{align*}
since $\tilde{Y}_u$ is stationary. Now, due to the third part of Lemma \ref{lemma:mixingaleandthetaweak}, we obtain
\begin{align*}
& \limsup_{N\rightarrow\infty } \sum_{k=1}^\infty \norm{\tilde{Y}_u(0)E[\tilde{Y}_u(k\delta)|\sigma(\{\tilde{Y}_u(j\delta):j\leq0,j\in\Z\}]}_{L^1}\\
&\quad \leq D \left(\sum_{h=1}^\infty(h+1)^{\frac{1}{\varepsilon}}\theta(\delta h)\right)^{\frac{\varepsilon}{1+\varepsilon}}<\infty,
\end{align*}
since \hyperlink{DD}{DD($\varepsilon$)} holds, using similar arguments as in (\ref{equation:seriesthetashifted}). By the dominated convergence theorem we obtain $V(K)\rightarrow0$ as $K\rightarrow\infty$. Finally, it is left to show (\ref{equation:CLTproofcondO13}). First, we note that
\begin{align*}
&\Bigg\lVert \frac{1}{t}\frac{\delta_N}{b_N}\sum_{i=1}^{\lfloor t(2m_N+1)\rfloor } K_{rect}\left(\frac{\tau_{i-m_N-1}^N-u}{b_N} \right)^2\tilde{Y}_{u}(N\tau_{i-m_N-1}^N)\tilde{Y}_{u}(N\tau_{i+k-m_N-1}^N)\\
&\quad-\frac{1}{t(2m_N+1)}\sum_{i=1}^{\lfloor t(2m_N+1)\rfloor }X_{i,2m_N+1}X_{i+k,2m_N+1}\Bigg\rVert_{L^1}\\
&\leq \frac{1}{2t}\frac{\delta_N}{b_N}E[|\tilde{Y}_{u}(0)\tilde{Y}_{u}(k\delta)|]\sum_{i=-m_N}^{m_N}\abs{K_{rect}\left(\frac{\tau_{i}^N-u}{b_N}\right)-K_{rect}\left(\frac{\tau_{i+k}^N-u}{b_N} \right) }\\
&\leq  \frac{1}{2t}\frac{\delta_N}{b_N}E[|\tilde{Y}_{u}(0)\tilde{Y}_{u}(k\delta)|] k B_K\underset{N\rightarrow\infty}{\longrightarrow}0
\end{align*}
for some constant $B_K\geq0$, since $K$ is of bounded variation. Similar arguments as in (\ref{eq:asymptoticvarianceCLTO2}) give 
\begin{align*}
\Bigg\lVert \frac{1}{t}\frac{\delta_N}{b_N}\sum_{i=1}^{\lfloor t(2m_N+1)\rfloor } K_{rect}\left(\frac{\tau_{i-m_N-1}^N-u}{b_N} \right)^2\tilde{Y}_{u}(N\tau_{i-m_N-1}^N)\tilde{Y}_{u}(N\tau_{i+k-m_N-1}^N)\\
-\frac{\delta_N}{b_N}\sum_{i=-m_N}^{m_N}\frac{1}{2}G_t\left(\frac{\tau_i^N-u}{b_N} \right)\tilde{Y}_{u}(N\tau_i^N)\tilde{Y}_{u}(N\tau_{i+k}^N)\Bigg\rVert_{L^1}\underset{N\rightarrow\infty}{\longrightarrow}0,
\end{align*}
where $G_t(x)$ is defined as in (\ref{equation:localizingkernelgtx}). Thus, (\ref{equation:CLTproofcondO13}) holds, since
\begin{align*}
\norm{\frac{\delta_N}{b_N}\sum_{i=-m_N}^{m_N}\frac{1}{2}G_t\left(\frac{\tau_i^N-u}{b_N} \right)\tilde{Y}_{u}(N\tau_i^N)\tilde{Y}_{u}(N\tau_{i+k}^N)-\frac{E[\tilde{Y}_u(0)\tilde{Y}_u(k\delta)]}{2}}_{L^1}\underset{N\rightarrow\infty}{\longrightarrow}0,
\end{align*}
by Lemma \ref{lemma:stationarylocalergodicdiscrete}. We conclude by similar steps as in (\ref{equation:fromL1toconvergenceindistribution}).
\end{enumerate}
\qed
%

\begin{Remark}
We restrict ourselves to the rectangular kernel (\ref{equation:rectangularkernel}) as for general kernels it seems only possible to establish the conditions (\ref{equation:CLTproofcond3}) and (\ref{equation:CLTproofcondO13}) for $t=1$, which is not sufficient to apply \cite[Theorem 2]{DM2002}.
\end{Remark}

\section{Locally stationary processes in continuous-time}
\label{sec4}
\noindent In the sequel we first review the definition of locally stationary processes in continuous-time from \cite{BS2021}. Then, we relate it to processes that possess a locally stationary approximation as introduced in Definition \ref{definition:statapproxconttime}.

\subsection{Preliminaries}
\label{sec4-1}
We summarize elementary properties of L\'evy processes and stochastic integration with respect to them. For further insight we refer to \cite{A2009} and \cite{S2013}.

\begin{Definition}
A real-valued stochastic process $L=\{L(t),t\in \R_0^+\}$ is called L\'evy process if
\begin{enumerate}[label={(\alph*)}]
\item $L(0)=0$ almost surely,
\item for any $n\in\N$ and $t_0<t_1<t_2<\dots<t_n$, the random variables $(L(t_0),L(t_1)-L(t_0),\dots,L(t_n)-L(t_{n-1}))$ are independent,
\item for all $s,t \geq 0$, the distribution of $L(s+t)-L(s)$ does not depend on $s$ and
\item $L$ is stochastically continuous.
\end{enumerate}
Without loss of generality we additionally consider $L$ to be c\`adl\`ag, i.e. right continuous with finite left limits.
\end{Definition}

Let $L=\{L(t),t\in \R_0^+\}$ be a real-valued L\'evy process. Then, $L(1)$ is an infinitely divisible real-valued random variable with characteristic triplet $(\gamma,\Sigma,\nu)$, where $\gamma \in \R$, $\Sigma>0$ and $\nu$ is a L\'evy measure on $\R$, i.e. $\nu(0)=0$ and $\int_{\R}\left(1\wedge\normabs{x}^2\right)\nu(dx)<\infty$. The characteristic function of $L(t)$ is given by
\begin{align}\label{eq:levykhintchine}
\begin{aligned}
\varphi_{L(t)}(z)&=E[e^{izL(t)}]=e^{t \Psi_{L}(z)},\\
\Psi_L(z)&=\left( i\gamma z -\frac{\Sigma z^2}{2}+\int_{\R } \left(e^{i z x}-1-i z x \mathbb{1}_{Z}(x)	\right)\nu(dx)\right),
\end{aligned}
\end{align}
where $z \in \R$ and $Z=\{x \in \R, \normabs{x}\leq 1\}$. 
%
If $\nu$ has finite second moment, i.e.
\begin{gather}\label{eq:condlevyproc}
\int_{\normabs{x}>1}\normabs{x}^2\nu(dx)<\infty \left(\iff\int_{\R}\normabs{x}^2\nu(dx)<\infty\right),
\end{gather}
then $L(t)\in L^2$ for all $t\geq0$ and we have $E[L(t)]=t \left(\gamma+\int_{\normabs{x}>1}x\nu(dx)\right)<\infty$ and $Var(L(t))=t \left(\Sigma +\int_{\R}x^2\nu(dx)\right)<\infty$. In the remainder we work with two-sided L\'evy process, i.e. $L(t)=L_1(t)\mathbb{1}_{\{t\geq0\}} - L_2(-t)\mathbb{1}_{\{t<0\}}$, where $L_1$ and $L_2$ are independent copies of a one-sided L\'evy process.\\
Consider
\begin{gather}\label{eq:stochint}
X(t)=\int_\R f(t,s) L(ds),
\end{gather}
where $t\in\R$ and $f:\R\times\R\mapsto\R$ is $\mathcal{B}(\R\times\R)-\mathcal{B}(\R)$ measurable.
Necessary and sufficient conditions for the stochastic integral (\ref{eq:stochint}) to exist are given in \cite[Theorem 3.3]{S2014}, namely if 
\begin{align}\label{eq:intexistencecond}
\begin{aligned}
&\Sigma\int_\R f(t,s)^2ds<\infty,\\
&\int_\R \int_{\R}\left((f(t,s)x)^2\wedge 1 \right)\nu(dx)ds<\infty\text{ and}\\
&\int_\R \normabs{f(t,s)\left( \gamma+\int_{\R}x\left(\mathbb{1}_{[0,1]} \left(\normabs{f(t,s)x}\right)- \mathbb{1}_{[0,1]}\left( \normabs{x}\right)\nu(dx)\right) \right)}ds<\infty
\end{aligned}
\end{align}
are satisfied, then (\ref{eq:stochint}) is well-defined. If $L$ satisfies (\ref{eq:condlevyproc}) and $f(t,\cdot)\in L^1(\R)\cap L^2(\R)$, then the conditions (\ref{eq:intexistencecond}) are satisfied and the integral $X(t)=\int_\R f(t,s)L(ds)$ exists in $L^2$. If $X=\{X(t), t\in\R\}$ with $X(t)$ as in (\ref{eq:stochint}) is well-defined, then $X(t)$ is infinitely divisible with characteristic triplet $(\gamma_{int},\Sigma_{int},\nu_{int})$, where
\begin{align}
\begin{aligned}\label{equation:charactersitictripletint}
\gamma_{int}&=\int_{\R}f(t,s)\gamma ds+ \int_\R\int_{\R}f(t,s)x\left(\mathbb{1}_{[0,1]}(\normabs{f(t,s)x})-\mathbb{1}_{[0,1]}(\normabs{x})\nu(dx)\right)ds,\\
\Sigma_{int}&=\Sigma\int_{\R}f(t,s)^2ds \quad\text{ and} \\
\nu_{int}(B)&= \int_{\R}\int_{\R} \mathbb{1}_B(f(t,s)x)\nu(dx) ds, 
\end{aligned}
\end{align}
where $B\in\mathcal{B}(\R)$.

\subsection{Locally stationary processes in continuous-time}
\label{sec4-2}

The following definition was first given in \cite[Definition 3.1]{BS2021}.

\begin{Definition}\label{definition:localstatconttime}
Let $Y_N=\{Y_N(t), t\in\R\}_{N\in\N}$ be a sequence of stochastic processes. Then, $Y_N$ is called locally stationary if there exists a representation
\begin{gather}\label{equation:localstatconttime}
Y_N(t) =\int_\R g_N(Nt,Nt-s)L(ds), 
\end{gather}
where
\begin{enumerate}[label={(\alph*)}]
\item $L$ is a two sided real-valued L\'evy process such that (\ref{eq:condlevyproc}) holds,
\item $g_N:\R\times\R\rightarrow\R$ with $g_N(Nt,\cdot)\in L^2(\R)$ for all $t\in\R$ and $N\in\N$,
\item there exists a limiting local kernel function $g:\R\times\R\rightarrow\R$ such that $g(t,\cdot)\in L^2(\R)$ for all $t\in\R$, $t\mapsto g(t,\cdot)$ is continuous for all $t\in\R$ and
\begin{align*}
g_N(Nt,\cdot)\underset{N\rightarrow \infty}{\overset{L^2}{\longrightarrow}}g(t,\cdot)
\end{align*}
for all $t\in\R$.
\end{enumerate}
\end{Definition}

\begin{Remark}
\begin{enumerate}[label={(\alph*)}]
\item Definition \ref{definition:localstatconttime} can also be reformulated in the frequency domain. To this end one defines locally stationary processes via a time-varying spectral representation, which can be locally approximated by the spectral representation of a stationary process (see \cite[Definition 3.2]{BS2021}). In the continuous-time framework, the definition of local stationarity in the time and frequency domain are equivalent (see \cite[Proposition 3.3]{BS2021}).
\item From \cite[Proposition 3.5]{BS2021} it follows that the sequence (\ref{equation:localstatconttime}) converges pointwise for each $t\in\R$ in distribution for $N\rightarrow\infty$ to the stationary process 
\begin{gather*}
\int_{\R}g(t,-s)L(ds).
\end{gather*}
\end{enumerate}
\end{Remark}

\begin{Proposition}
Consider a sequence of stochastic processes $Y_N$ in the form of (\ref{equation:localstatconttime}) such that (\ref{eq:condlevyproc}) holds and $E[L(1)]=0$. Let $\tilde{Y}_u=\{\tilde{Y}_u(t),t\in\R\}_{u\in\R^+}$ be the family of stationary processes $\tilde{Y}_u(t)=\int_\R g(u,t-s)L(ds)$, such that $g_N(Nt,\cdot),g(t,\cdot) \in L^2(\R)$ for all $t\in\R$ and $N\in\N$. Assume that $\tilde{Y}_u$ is a locally stationary approximation of $Y_N$ for some $p\geq1$. Then, $Y_N$ is locally stationary.
\end{Proposition}
\begin{proof}
It holds
\begin{align*}
\Lnorm{g_N(Nt,\cdot)-g(t,\cdot)}&=\Lnorm{g_N(Nt,Nt-\cdot)-g(t,Nt-\cdot)}= \frac{1}{\Sigma}\Lnorm{Y_N(t)-\tilde{Y}_t(Nt)}\\
&\leq  \frac{1}{\Sigma} \frac{C}{N}\underset{N\rightarrow\infty}{\longrightarrow}0.
\end{align*}
\end{proof}

\begin{Proposition}
Let $Y_N=\{Y_N(t), t\in\R\}_{N\in\N}$ be a locally stationary process. Define $\tilde{Y}_u(t)=\int_\R g(u,t-s)L(ds)$, where $g$ is the corresponding limiting local kernel function of $Y_N$. If it holds that $\sup_{u\in\R}\norm{g(u,\cdot)}_{L^2}<\infty$,
\begin{align*}
\norm{g(u,\cdot)-g(v,\cdot)}_{L^2}\leq C \abs{u-v} \text{ and } \norm{g(Nt,\cdot)-g(t,\cdot)}_{L^2} \leq C \frac{1}{N},
\end{align*}
then $\tilde{Y}_u$ is a locally stationary approximation of $Y_N$ for $p=2$.
\end{Proposition}
\begin{proof}
The ergodicity of $\tilde{Y}_u$ follows for instance from \cite[Theorem 3.5]{FS2013}. Then it is sufficient to observe that
\begin{align*}
\norm{\tilde{Y}_u(t)-\tilde{Y}_v(t)}_{L^2}&=\Sigma \norm{g(u,\cdot)-g(v,\cdot)}_{L^2}\leq C \Sigma \abs{u-v} \text{ and}\\
\norm{Y_N(t)-\tilde{Y}_{t}(Nt)}_{L^2}& = \Sigma \norm{g_N(Nt,Nt-\cdot)- g(t,Nt-\cdot)}_{L^2}\leq \frac{C \Sigma }{N}.
\end{align*}
\end{proof}

\section{Time-varying L\'evy-driven state space models and CARMA(p,q) processes}
\label{sec5}
In this section we first introduce time-varying L\'evy-driven autoregressive moving average (tvCARMA) processes. Then we review conditions such that sequences of tvCARMA processes are locally stationary and discuss conditions for the existence of a locally stationary approximation. Since the class of time-varying L\'evy-driven CARMA processes is embedded in the class of L\'evy-driven linear state space models, we state our results in terms of this class of processes.
Throughout this section, $L=\{L(t),t\in\R\}$ denotes a two-sided L\'evy process with values in $\R$, where we assume that the characteristic triplet $(\gamma,\Sigma,\nu)$ satisfies (\ref{eq:condlevyproc}).

\subsection{L\'evy-driven Ornstein-Uhlenbeck (CAR(1)) processes}
\label{sec5-1}
The simplest L\'evy-driven CARMA(p,q) process is obtained for $p=1$ and $q=0$, i.e. the L\'evy-driven CAR(1) or L\'evy-driven Ornstein-Uhlenbeck process. It is given as the stationary solution to the differential equation $dY(t)=-aY(t)dt+L(dt)$ for a constant $a>0$. The solution can be expressed as 
\begin{gather*}
Y(t)=\int_{-\infty}^t e^{-a(t-s)} L(ds).
\end{gather*}
If one allows $a$ to be time-varying, we arrive at the so called time-varying L\'evy-driven CAR(1) process, which is given by
\begin{gather*}
Y(t)=\int_{-\infty}^t e^{\int_s^ta(\tau)d\tau} L(ds).
\end{gather*}
To investigate processes of the above form in the context of local stationarity, we consider $Y_N=\{Y_N(t),t\in\R\}_{N\in\N}$ to be a sequence of rescaled time-varying CAR(1) processes defined by 
\begin{align}\label{eq:tvcar}
\begin{aligned}
Y_N(t)&=\int_{-\infty}^{\infty}g_N(Nt,Nt-s)L(ds), \text{ with kernel function}\\
g_N(Nt,Nt-s)&=\mathbb{1}_{\{Nt-s\geq0\}} e^{-\int_{s}^{Nt}a\left(\frac{\tau}{N}\right)d\tau}= \mathbb{1}_{\{Nt-s\geq0\}} e^{-\int_{-(Nt-s)}^0a\left(\frac{\tau+Nt}{N}\right)d\tau},
\end{aligned}
\end{align}
where $a:\R\rightarrow\R_0^+$ is continuous such that $s\mapsto e^{-\int_{-s}^{0}a\left(\frac{\tau+Nt}{N} \right)d\tau}\in L^1(\R^+)$ for all $t\in\R$ and $N\in\N$, which ensures the existence of (\ref{eq:tvcar}), since additionally (\ref{eq:condlevyproc}) holds. In the next proposition we revise sufficient conditions such that the sequence $Y_N$ is locally stationary in the sense of Definition \ref{definition:localstatconttime}. Then we give additional conditions for the existence of a locally stationary approximation for $p=2$ and $p=4$, where the locally stationary approximation $\tilde{Y}_{u}$ is given by
\begin{align}\label{eq:locapproxcar}
\begin{aligned}
\tilde{Y}_{u}(t)&=\int_\R g(u,t-s)L(ds), \text{ with kernel function}\\
g(u,t-s)&= \mathbb{1}_{\{t-s\geq0\}} e^{-a(u)(t-s)},
\end{aligned}
\end{align}
where we assume that $s\mapsto \mathbb{1}_{\{s\geq0\}} e^{-a(u)s}\in L^1(\R^+)$ for all $u\in\R^+$ and $N\in\N$.

\begin{Proposition}[{\cite[Proposition 4.1]{BS2021}}]
Let $Y_N$ be a sequence of time-varying L\'evy-driven CAR(1) processes as given in (\ref{eq:tvcar}). 
Then, $Y_N$ is locally stationary if 
\begin{enumerate}[label={(\alph*)}]
\item the coefficient function $a$ is continuous and
\item for every $T\in\R^+$ there exists $\varepsilon_T>0$ such that $a(s)>\varepsilon_T$ for all $s\leq T$.
\end{enumerate}
The corresponding limiting local kernel function is given by $g(u,t)$ as defined in (\ref{eq:locapproxcar}).
\end{Proposition}

%

\begin{Lemma}\label{lemma:4thmoment}
For fixed $u\in\R$ let $f:\R\times\R\rightarrow\R$ be a $\mathcal{B}(\R\times\R)-\mathcal{B}(\R)$ measurable function such that $f(u,\cdot) \in  L^1(\R)\cap L^4(\R)$ and let $L$ be a two sided L\'evy-process such that $\int_\R x^4\nu(dx)<\infty$. Then, $\int_{-\infty}^t f(u,t-s)L(ds)\in L^4$ and
\begin{align*}
E\left[\left(\int_{-\infty}^t f(u,t-s)L(ds)\right)^4\right]
=&\int_\R f(u,s)^4ds \int_\R x^4\nu(dx)+3\Sigma_L^2\left(\int_\R f(u,s)^2ds\right)^2\\
&+\mu_L^4 \left(\int_\R f(u,s)ds\right)^4+6\mu_L^2 \Sigma_L\left(\int_\R f(u,s)ds \right)^2 \int_\R f(u,s)^2ds\\
&+4\mu_L\!\! \left( \int_\R f(u,s)^3ds\int_\R x^3\nu(dx)\int_\R f(u,s)ds\right),
\end{align*}
where $\Sigma_L=\Sigma+\int_{\R}x^2\nu(dx)$ and $\mu_L=\gamma+\int_{|x|>1}x\nu(dx)$.
\end{Lemma}
\begin{proof}
The finiteness of the fourth moments follows e.g. from \cite{S2013}. The characteristic function of $\int_{-\infty}^t f(u,t-s)L(ds)$ is known from (\ref{eq:levykhintchine}). Usual calculations then give the above representation. Note that (\ref{equation:charactersitictripletint}) connects the characteristic triplet of $\int_{-\infty}^t f(u,t-s)L(ds)$ with the characteristic triplet of $L$.
\end{proof}

\begin{Proposition}\label{proposition:car1locstat}
Let $Y_N$ be a sequence of time-varying L\'evy-driven CAR(1) processes as given in (\ref{eq:tvcar}). 
Then, $\tilde{Y}_{u}$ as given in (\ref{eq:locapproxcar}) is a locally stationary approximation of $Y_N$ for $p=2$, if
\begin{enumerate}[label={(\alph*)}]
\item the coefficient function $a$ is Lipschitz with constant L and
\item $\inf_{s\in\R}a(s)>0$.
\end{enumerate}
If additionally 
\begin{enumerate}[label={(\alph*)}]
\item[(c)] $\int_{\R}x^4\nu(dx)<\infty$,
\end{enumerate}
then $\tilde{Y}_{u}$ is also a locally stationary approximation of $Y_N$ for $p=4$.
\end{Proposition}

\begin{proof}
It is clear that $\tilde{Y}_{u}$ is stationary for all $u\in\R^+$. From \cite[Theorem 3.5]{FS2013} it also follows that $\tilde{Y}_{u}$ is ergodic. Let $\Sigma_L=\Sigma+\int_{\R}x^2\nu(dx)$ and $\mu_L=\gamma+\int_{|x|>1}x\nu(dx)$. Now, for $\varepsilon=\inf_{s\in\R}a(s)>0$ it holds
\begin{gather*}
\Lnorm{\tilde{Y}_{u}(t)}^2=\Sigma_L\int_\R g\left(u,t-s\right)^2ds+\mu_L^2\left(\int_\R g\left(u,t-s\right)ds\right)^2=\frac{\Sigma_L}{2a(u)}+\frac{\mu_L^2}{a(u)^2}\leq \frac{\Sigma_L}{2\varepsilon}+\frac{\mu_L^2}{\varepsilon^2}<\infty.
\end{gather*}
For $u,v\in\R$ we obtain
\begin{align*}
\norm{ \tilde{Y}_{u}(t)-\tilde{Y}_{v}(t)}_{L^2}^2&=\Sigma_L\norm{ g\left(u,t-\cdot\right)-g\left(v,t-\cdot\right)}_{L^2}^2+\mu_L^2\left(\int_\R g\left(u,t-\cdot\right)-g\left(v,t-\cdot\right)ds\right)^2\\
&=\Sigma_L\int_{0}^\infty \left(e^{-a(u)s}-e^{-a(v)s}\right)^2ds+\mu_L^2\left(\int_{0}^\infty e^{-a(u)s}-e^{-a(v)s}ds\right)^2=:P_1+P_2.
\end{align*}
Then,
\begin{align*}
P_1&=\Sigma_L\int_{0}^\infty e^{-s\varepsilon}\left(e^{-(a(u)s-\frac{\varepsilon}{2}s)}-e^{-(a(v)s-\frac{\varepsilon}{2}s)}\right)^2ds\\
&\leq
\Sigma_L\int_{0}^\infty e^{-s\varepsilon}s^2 (a(u)-a(v))^2ds\leq
C_1^2 (u-v)^2
\end{align*}
for a constant $C_1>0$, since $a(\cdot)$ is Lipschitz and $x\mapsto e^{-x}$ is Lipschitz on $\R_0^+$ with constant $1$. For $P_2$ it holds analogously 
\begin{gather*}
P_2\leq\mu_L^2\left(\int_{0}^\infty e^{-a(u)s}-e^{-a(v)s}ds\right)^2
\leq C_2^2 (u-v)^2
\end{gather*}
for a constant $C_2>0$. Therefore,
\begin{gather*}
\norm{ \tilde{Y}_{u}(t)-\tilde{Y}_{v}(t)}_{L^2}\leq \sqrt{C_1^2+C_2^2} (u-v).
\end{gather*}
Second,
\begin{align*}
\Lnorm{Y_N(t)-\tilde{Y}_t(Nt)}^2\!\!&=\Sigma_L\Lnorm{ g_N\left(Nt,Nt-\cdot\right)-g\left(t,Nt-\cdot\right)}^2\\
&\quad+\mu_L^2\left(\int_\R g_N\left(Nt,Nt-u\right)-g\left(t,Nt-u\right)du\right)^2\\
&=\Sigma_L\!\int_0^\infty\!\! \left(\!e^{-\int_{-s}^0a(\frac{\tau+Nt}{N})d\tau}\!-\!e^{-a(t)s}\right)^2\!\!ds\!+\!\mu_L^2\!\left(\int_\R \!e^{-\int_{-s}^0a(\frac{\tau+Nt}{N})d\tau}\!-\!e^{-a(t)s}ds\!\right)^2\\
&=:P_3+P_4.
\end{align*}
Then,
\begin{align*}
P_3=&\Sigma_L\int_0^\infty e^{-s\varepsilon}\left(e^{-\int_{-s}^0(a(\frac{\tau}{N}+t)-\frac{\varepsilon}{2})d\tau}-e^{-(a(t)-\frac{\varepsilon}{2})s}\right)^2ds\\
\leq&\Sigma_L\int_0^\infty \!\!e^{-s\varepsilon}\left(\!- \int_{-s}^0\!a\left(\frac{\tau}{N}+t\right)d\tau+a(t)s\right)^2ds
\leq\Sigma_L\!\int_0^\infty e^{-s\varepsilon} L^2\left( \int_{-s}^0\frac{|\tau|}{N}d\tau\right)^2ds
=\!\frac{C_3^2}{N^2}
\end{align*}
for a constant $C_3>0$, since $a(\cdot)$ is Lipschitz with constant $L$ and $x\mapsto e^{-x}$ is Lipschitz continuous on $\R_0^+$ with Lipschitz constant $1$. 
Analogously, 
\begin{gather*}
P_4\leq\mu_L^2\left(\int_\R \left|e^{-\int_{-s}^0a(\frac{\tau+Nt}{N})d\tau}-e^{-a(t)s}\right|ds\right)^2
\leq \frac{C_4^2}{N^2}
\end{gather*}
for a constant $C_4>0$.
Overall,
\begin{gather*}
\Lnorm{Y_N(t)-\tilde{Y}_t(Nt)}\leq \frac{1}{N}\sqrt{C_3^2+C_4^2}.
\end{gather*}
Since $\int_{\R}x^4\nu(dx)<\infty$, Lemma \ref{lemma:4thmoment} gives a closed form expression for $\lVert Y_N(t)-\tilde{Y}_t(Nt)\rVert_{L^4}^4$ and $\lVert \tilde{Y}_{u}(t)-\tilde{Y}_{v}(t)\rVert_{L^4}^4$. Then, calculations analogous to the steps for $p=2$ show that $\tilde{Y}_u$ is a locally stationary approximation of $Y_N$ for $p=4$.
\end{proof}

In the following theorem we apply the results from Section \ref{sec3} to time-varying L\'evy-driven CAR(1) processes and provide asymptotic results for different sample moments and observations. \\
In particular, these results provide a first step towards a method of moments based estimation procedure including consistency, as well as results on the asymptotic distribution. 
 
\begin{Theorem}\label{theorem:statisticforcar1}
Let $Y_N$ be a sequence of time-varying L\'evy-driven CAR(1) processes as given in (\ref{eq:tvcar}) and assume that
\begin{enumerate}[label={(\alph*)}]
\item[(A)] the conditions $(a)$ and $(b)$ from Proposition \ref{proposition:car1locstat} hold. 
\end{enumerate}
Then, for observations sampled according to \hyperref[observations:O1]{(O1)} or \hyperref[observations:O2]{(O2)}, we obtain for all $u\in\R^+$ and $k\in\N_0$
\begin{align}
\frac{\delta_N}{b_N} \sum_{i=-m_N}^{m_N}K\left(\frac{\tau_i^N-u}{b_N} \right)Y_N(\tau_i^N) &\overset{L^1}{\underset{N\rightarrow\infty}{\longrightarrow}}E[\tilde{Y}_u(0)]\text{ and for \hyperref[observations:O1]{(O1)} also}\label{equation:car1samplemean}\\
\frac{\delta_N}{b_N} \sum_{i=-m_N}^{m_N}K\left(\frac{\tau_i^N-u}{b_N} \right)Y_N(\tau_i^N)Y_N\left(\tau_i^N+\frac{k}{N}\right) &\overset{L^1}{\underset{N\rightarrow\infty}{\longrightarrow}}E[\tilde{Y}_u(0)\tilde{Y}_u(k)]\label{equation:car1samplecov}
\end{align}
with $\tilde{Y}_{u}$ as given in (\ref{eq:locapproxcar}). If we additionally assume that 
\begin{enumerate}[label={(\alph*)}]
\setcounter{enumi}{2}
\item[(B)] $\int_{|x|>1}|x|^{2+\varepsilon}\nu(dx)<\infty$ for some $\varepsilon>0$,
\end{enumerate}
then, the convergence in (\ref{equation:car1samplecov}) also holds for \hyperref[observations:O2]{(O2)}. Now, if in addition
\begin{enumerate}[label={(\alph*)}]
\setcounter{enumi}{2}
\item[(C)] $E[L(1)]=0$,
\item[(D)] $\sqrt{m_N}b_N\rightarrow0$, as $N\rightarrow\infty$ and 
\item[(E)] $\sigma(u)^2>0$, where $\sigma(u)^2=\begin{cases}\frac{1}{2}E[\tilde{Y}_u(0)^2]+\sum_{h=1}^\infty E[\tilde{Y}_u(0)\tilde{Y}_u(h)] , &\text{ if \hyperref[observations:O1]{(O1)} holds,}\\
E[\tilde{Y}_u(0)^2],&\text{ if \hyperref[observations:O2]{(O2)} holds,} \end{cases}$
\end{enumerate}
then $\sigma(u)^2<\infty$ and
\begin{gather}\label{equation:car1cltsamplemean}
\sqrt{\frac{\delta_N}{b_N}} \sum_{i=-m_N}^{m_N}K_{rect}\left(\frac{\tau_i^N-u}{b_N}\right)Y_N(\tau_i^N)\overset{d}{\underset{N\rightarrow\infty}{\longrightarrow}} \sigma(u)^2 Z,
\end{gather}
where $Z$ is a standard normally distributed random variable and $K_{rect}$ as defined in (\ref{equation:rectangularkernel}). Finally, if we additionally assume that
\begin{enumerate}[label={(\alph*)}]
\setcounter{enumi}{5}
\item[(F)] $\int_{|x|>1}|x|^{4+\varepsilon}\nu(dx)<\infty$ for some $\varepsilon>0$ \text{ and}
\item[(G)] 
$\tilde{\sigma}(u)^2>0$, where \\
$\tilde{\sigma}(u)^2=\begin{cases}\frac{1}{2}E[\tilde{Y}_u(0)^2\tilde{Y}_u(k)^2]+\sum_{h=1}^\infty E[\tilde{Y}_u(0)\tilde{Y}_u(k)\tilde{Y}_u(h)\tilde{Y}_u(h+k)] , &\text{ if \hyperref[observations:O1]{(O1)} holds,}\\
E[\tilde{Y}_u(0)^2\tilde{Y}_u(k)^2],&\text{ if \hyperref[observations:O2]{(O2)} holds,} \end{cases}$ 
\end{enumerate}
then $\tilde{\sigma}(u)^2<\infty$ and
\begin{align}\label{equation:car1cltsamplecov}
\begin{aligned}
&\sqrt{\frac{\delta_N}{b_N}}\sum_{i=-m_N}^{m_N}K_{rect}\left(\frac{\tau_i^N-u}{b_N}\right) \left(Y_N(\tau_i^N)Y_N\left(\tau_i^N+\frac{k}{N}\right)-Cov(\tilde{Y}_u(0)\tilde{Y}_u(k))\right)\\
&\qquad\qquad\qquad\qquad\qquad\qquad\qquad\overset{d}{\underset{N\rightarrow\infty}{\longrightarrow}} \tilde{\sigma}(u)^2 Z.
\end{aligned}
\end{align}
\end{Theorem}
\begin{proof}
From Proposition \ref{proposition:car1locstat} it follows that $\tilde{Y}_{u}$ is a locally stationary approximation of $Y_N$ for $p=2$. If \hyperref[observations:O1]{(O1)} holds, Theorem \ref{theorem:lawoflargenumbersO1} implies the convergence in (\ref{equation:car1samplemean}). Analogous to \cite[Corollary 3.4]{CS2018} and since $g(u,\cdot)\in L^{2}(\R)$ for all $u\in\R$, it follows that $\tilde{Y}_u$ is $\theta$-weakly dependent with $\theta$-coefficients $\theta(h)$. Therefore, if \hyperref[observations:O2]{(O2)} holds, we obtain the convergence in (\ref{equation:car1samplemean}) from part (a) of Theorem \ref{theorem:lawoflargenumbersdiscreteobs}.\\ 
Define $g:\R^{k+1}\rightarrow\R$ as $g(x_1,\ldots,x_{k+1})=x_1x_{k+1}$, $\tilde{Z}_u^k(t)$ and $Z_N^k(t)$ as in Section \ref{sec2-3}. As described in Remark \ref{remark:inheritancepolynomials} we have $g\in\mathcal{L}_{k+1}(1,1)$. From Proposition \ref{proposition:inheritanceproperties} part (a) it follows that $g(\tilde{Z}_u^k(t))=\tilde{Y}_u(t)\tilde{Y}_u(t-k)$ is a locally stationary approximation of $g(Z_N^k(t))=Y_N(t)Y_N(t-\frac{k}{N})$ for $p=1$. If \hyperref[observations:O1]{(O1)} holds, the convergence in (\ref{equation:car1samplecov}) then follows again from Theorem \ref{theorem:lawoflargenumbersO1}. In order to prove (\ref{equation:car1samplecov}) for \hyperref[observations:O2]{(O2)}, we first note that $E[|Y_u(0)|^{2+\varepsilon}]<\infty$, since $\int_{|x|>1}|x|^{2+\varepsilon}\nu(dx)<\infty$ and $g(u,\cdot)\in L^{2+\varepsilon}(\R)$ for all $u\in\R^+$. Proposition \ref{proposition:inheritanceproperties} part (c) implies that $g(\tilde{Z}_u^k(t))=\tilde{Y}_u(t)\tilde{Y}_u(t-k)$ is $\theta$-weakly dependent. Hence, part (a) of Theorem \ref{theorem:lawoflargenumbersdiscreteobs} gives the stated convergence in (\ref{equation:car1samplecov}) for \hyperref[observations:O2]{(O2)}.
 \\
Analogous to the proof of \cite[Theorem 3.36]{CSS2020}, one can show that $\tilde{Y}_u$ has exponentially decaying $\theta$-coefficients, such that \hyperlink{DD}{DD($\varepsilon$)} holds for any $\varepsilon>0$. Theorem \ref{theorem:centrallimit} then implies the convergence as stated in (\ref{equation:car1cltsamplemean}) for both \hyperref[observations:O1]{(O1)} and \hyperref[observations:O2]{(O2)}. \\
Since $\int_{|x|>1}|x|^{4+\varepsilon}\nu(dx)<\infty$ and $g(u,\cdot)\in L^{4+\varepsilon}(\R)$ for all $u\in\R^+$, we have $E[|\tilde{Y}_u(0)|^{4+\varepsilon}]<\infty$ for all $u\in\R^+$.
From Proposition \ref{proposition:car1locstat} it follows that $\tilde{Y}_{u}$ is a locally stationary approximation of $Y_N$ for $p=4$, such that part (c) of Proposition \ref{proposition:inheritanceproperties} implies that $g(\tilde{Z}_u^k(t))$ is a locally stationary approximation of $g(Z_N^k(t))$ for $p=2$, where $g(\tilde{Z}_u^k(t))=\tilde{Y}_u(t)\tilde{Y}_u(t-k)$ is $\theta$-weakly dependent with $\theta$-coefficients $\theta_{\tilde{Z}_u^k}(h)\in\mathcal{O}\left(\theta(h)^{\frac{\varepsilon}{1+\varepsilon}}\right)$. Since $\tilde{Y}_u$ has exponentially decaying $\theta$-coefficients $\theta(h)$, $\theta_{\tilde{Z}_u^k}(h)$ satisfies \hyperlink{DD}{DD($\varepsilon$)} for any $\varepsilon>0$. Theorem \ref{theorem:centrallimit} then implies (\ref{equation:car1cltsamplecov}) for \hyperref[observations:O1]{(O1)} and \hyperref[observations:O2]{(O2)}.
\end{proof}

\subsection{L\'evy-driven state space and CARMA(p,q) models}
\label{sec5-2}
Before we give the definition of general time-varying L\'evy-driven state space models, we introduce time-varying L\'evy-driven CARMA processes (see e.g. \cite{B2014} for an introduction to time-invariant L\'evy-driven CARMA processes).
For positive integers $p>q$, let $a_i(t)$, $i=1,\ldots,p$ and $b_j(t)$, $j=0,\ldots,q$ be continuous real functions. Then, the polynomials 
\begin{align*}
p(t,z)&=z^p+a_1(t)z^{p-1}+\ldots+a_{p-1}(t)z+a_p(t)\text{ and}\\
q(t,z)&=b_0(t)+b_1(t)z+\ldots+b_{q-1}(t)z^{q-1}+b_q(t)z^q,
\end{align*}
are respectively called autoregressive (AR) and moving average (MA) polynomials. A time-varying L\'evy-driven CARMA process is then defined as the solution to the formal differential equation
\begin{gather}\label{eq:tvCARMAformaldiff}
p(t,D)Y(t)=q(t,D)DL(t),
\end{gather}
where $D$ denotes the differential operator with respect to time. 
Since L\'evy processes are in general not differentiable, we interpret (\ref{eq:tvCARMAformaldiff}) as usual to be equivalent to the state space representation
\begin{align}\label{eq:tvCARMAstatespace}
\begin{aligned}
Y(t)&=B(t)'X(t)\text{ and}\\
dX(t)&=A(t)X(t)dt+CL(dt),
\end{aligned}
\end{align}
where $A(t)$ is in companion form, i.e.
\begin{align*}
A(t)&=\left( \begin{array}{cccc}
0 & 1 & \ldots & 0 \\
\vdots & ~ & \ddots & \vdots \\
0 & ~ & ~ & 1 \\
-a_p(t) & -a_{p-1}(t) & \ldots & -a_1(t) \\
\end{array}\right)\in M_{p\times p}(\R)\text{ and}\\
B(t)&=\left( \begin{array}{c}
b_0(t)  \\
 b_1(t)\\
\vdots \\
b_{p-1}(t)\\
\end{array}\right)\in M_{p\times 1}(\R), \qquad 
C=\left( \begin{array}{c}
0  \\
\vdots\\
0 \\
1\\
\end{array}\right)\in M_{p\times 1}(\R),\quad t\in\R.
\end{align*}
For a discussion on state space representations of the form of (\ref{eq:tvCARMAstatespace}) with a Brownian motion as driving noise we refer to \cite[Section 2.1.1]{S2010}. As explained in \cite{BS2021}, there exists a solution to the above state space representation which is given by
\begin{gather}\label{eq:tvCARMAsolution}
X(t)=\int_{-\infty}^t\Psi(t,s)C L(ds)\text{ and } Y(t)=B(t)' \int_{-\infty}^t\Psi(t,s)C L(ds)
\end{gather}
for $t\in\R$, provided that the integrals exist in $L^2$. The matrix $\Psi(t,t_0)$ for $t>t_0$ is the unique solution to the homogeneous initial value problem (IVP)
\begin{gather}\label{eq:ivpA}
\begin{gathered}
\frac{d}{dt}\Psi(t,t_0)=A(t)\Psi(t,t_0)\quad \text{ with initial condition}\quad\Psi(t_0,t_0)=\mathbf{1}_p. 
\end{gathered}
\end{gather}
In particular, it holds $\Psi(t,t_0)=\Psi(t,s)\Psi(s,t_0)$ for $t_0<s<t$. For an extensive collection of results on the IVP (\ref{eq:ivpA}) we refer to \cite[Section 3 and 4]{B1970}. 

\begin{Definition}
Let $Y=\{Y(t),t\in\R\}$ be a solution to the state space representation (\ref{eq:tvCARMAstatespace}) in the form of (\ref{eq:tvCARMAsolution}). Then, $Y$ is referred to as a time-varying L\'evy-driven CARMA(p,q) process.
\end{Definition}

The IVP (\ref{eq:ivpA}) is solved by the Peano-Baker series 
\begin{gather*}
\Psi(t,t_0)=\mathbf{1}_p+\int_{t_0}^tA(\tau_1)d\tau_1+\int_{t_0}^tA(\tau_1)\int_{t_0}^{\tau_1}A(\tau_2)d\tau_2d\tau_1+\ldots=\sum_{n=0}^\infty \mathcal{I}_n,
\end{gather*}
where $\mathcal{I}_0=\mathbf{1}_p$ and $\mathcal{I}_n=\int_{t_0}^tA(\tau_1)\int_{t_0}^{\tau_1}A(\tau_2)\ldots\int_{t_0}^{\tau_{n-1}}A(\tau_n)d\tau_n\ldots d\tau_2d\tau_1$. Considering \cite[Remark 2]{BS2011b}, the condition
\begin{gather}\label{eq:commcond}
\left[A(t),\int_{t_0}^tA(s)ds\right]=0
\end{gather}
for all $t>t_0$, i.e. if $A(t)$ and $\int_{t_0}^tA(s)ds$ commute, ensures that the Peano-Baker Series can be expressed as
\begin{gather*}
\Psi(t,t_0)=\sum_{n=0}^\infty \frac{1}{n!}\left(\int_{t_0}^tA(\tau)d\tau\right)^n=e^{\int_{t_0}^tA(\tau)d\tau}.
\end{gather*}
If (\ref{eq:commcond}) holds, for $t>t_0$ the solution (\ref{eq:tvCARMAsolution}) of the state space representation (\ref{eq:tvCARMAstatespace}) simplifies to
\begin{align*}
\begin{aligned}
X(t)&=e^{\int_{t_0}^tA(\tau)d\tau}X(t_0)+\int_{t_0}^te^{\int_{s}^tA(\tau)d\tau} C L(ds) =\int_{-\infty}^te^{\int_{s}^tA(\tau)d\tau} C L(ds) \quad\text{ and}\\ 
Y(t)&=B(t)'e^{\int_{t_0}^tA(\tau)d\tau}X(t_0)+\int_{t_0}^tB(t)'e^{\int_{s}^tA(\tau)d\tau} C L(ds) =\int_{-\infty}^tB(t)'e^{\int_{s}^tA(\tau)d\tau} C L(ds).
\end{aligned}
\end{align*}

\begin{Remark}
If $[A(t),A(s)]=0$ for all $s,t\in\R$, i.e. if $A(s)$ and $A(t)$ commute, the commutativity assumption (\ref{eq:commcond}) clearly holds. As a matter of fact, the statements are even equivalent (see \cite[Exercise 4.8]{R1996}).
Further discussions on the commutativity condition (\ref{eq:commcond}) can be found for instance in \cite{WS1976}.
\end{Remark}

In the following we consider matrix functions $A(t)$, which are not necessarily in companion form and additionally allow $C$ to be time-varying, i.e. $C(t)\in M_{p\times1}(\R)$, $t\in\R$. This leads to the observation and state equation
\begin{align}\label{eq:tvLDstatespace}
\begin{aligned}
Y(t)&=B(t)'X(t) \quad\text{ and}\\
dX(t)&=A(t)X(t)dt+C(t)L(dt),
\end{aligned}
\end{align}
where $A(t)\in M_{p\times p}(\R)$ and $B(t),C(t)\in M_{p\times1}(\R)$, $t\in\R$ are arbitrary continuous coefficient functions. A solution of (\ref{eq:tvLDstatespace}) is given by (see \cite[Section 4]{BS2021})
\begin{gather}\label{eq:tvLDstatespacesolution}
X(t)=\int_{-\infty}^t\Psi(t,s)C(s)L(ds)\text{ and }Y(t)=B(t)\int_{-\infty}^t\Psi(t,s)C(s)L(ds),
\end{gather} 
provided that the integrals exist in $L^2$. Again, $\Psi(t,t_0)$ is the unique matrix solution of the IVP (\ref{eq:ivpA}).

\begin{Definition}
Let $Y=\{Y(t),t\in\R\}$ be a solution to the state space representation (\ref{eq:tvLDstatespace}) in the form of (\ref{eq:tvLDstatespacesolution}). Then, $Y$ is referred to as a time-varying L\'evy-driven state space process.
\end{Definition}

For an initial time $t_0\in\R$ a time-varying L\'evy-driven state space process can be expressed as
\begin{gather*}
X(t)=\Psi(t,t_0)\left(X(t_0)+\int_{t_0}^t\Psi(s,t_0)^{-1}C(s)L(ds) \right).
\end{gather*}

\begin{Remark}
In \cite[Corollary 3.4]{SS2012}, the authors showed equivalence of the class of time-invariant L\'evy-driven CARMA processes and time-invariant L\'evy-driven state space processes. The relation of the above classes in the time-varying setting has recently been investigated in \cite{BS2021}. While the classes of time-varying L\'evy-driven CARMA processes and time-varying L\'evy-driven state space process are not necessarily equivalent for non-continuous coefficient functions (see \cite[Proposition 4.6]{BS2021}), the authors showed that a state space process of the form of (\ref{eq:tvLDstatespace}), whose coefficient functions $A$ and $C$ are sufficiently often differentiable, is equivalent to a time-varying L\'evy-driven CARMA process of the form (\ref{eq:tvCARMAsolution}) if and only if it is instantaneously controllable \cite[Proposition 4.8]{BS2021}. For more information on instantaneous controllability we refer to \cite[Chapter 9 and 10]{R1996}.
\end{Remark}

The following definition from \cite[Chapter 6]{R1996} directly ensures the existence of the stochastic integrals in (\ref{eq:tvCARMAsolution}) and (\ref{eq:tvLDstatespacesolution}).

\begin{Definition}
A linear state space model of the form of (\ref{eq:tvLDstatespace}) is called uniformly exponentially stable if there exists $\gamma>0$ and $\lambda>0$ such that
\begin{gather*}
\norm{\Psi(t,t_0)}\leq\gamma e^{-\lambda(t-t_0)}
\end{gather*}
for all $t>t_0$.
\end{Definition}

\begin{Remark}
If a linear state space model of the form (\ref{eq:tvCARMAstatespace}) or (\ref{eq:tvLDstatespace}) is uniformly exponential stable, then the integrals in (\ref{eq:tvCARMAsolution}) or (\ref{eq:tvLDstatespacesolution}) (if $C$ is additionally uniformly bounded) are well-defined, since we assumed that additionally (\ref{eq:condlevyproc}) holds.  
\end{Remark}

\subsection{Time-varying L\'evy-driven state space models}
\label{sec5-3}

First, we review conditions from \cite{BS2021} that are sufficient for a sequence of time-varying L\'evy-driven state space models to be locally stationary in the sense of Definition \ref{definition:localstatconttime}. Then we investigate suitable conditions for these models to possess a locally stationary approximation. In fact, it will turn out that both conditions show high similarity, as we have already seen it in Section \ref{sec5-1}\\
To be able to establish local stationarity, we consider a sequence $Y_N$ of time-varying L\'evy-driven state space models, i.e.
\begin{align}\label{eq:seqtvLDstatespacesolution}
\begin{aligned}
Y_N(t)&=\int_\R g_N(Nt,Nt-s)L(ds), \text{ with kernel function}\\
g_N(Nt,Nt-s)&= \mathbb{1}_{\{Nt-s\geq0\}} B(t)' \Psi_{N,t}(0,-(Nt-s))C\left(\frac{-(Nt-s)}{N}+t \right),
\end{aligned}
\end{align}
where $A(t)\in M_{p\times p}(\R)$ and $B(t),C(t)\in M_{p\times1}(\R)$, $t\in\R$,  are arbitrary continuous coefficient functions and $\Psi_{N,t}(0,-(Nt-u))$ is the solution to the IVP
\begin{gather*}
\Psi_{N,t}(s_0,s_0)= \mathbf{1}_p,\qquad
\frac{d}{ds}\Psi_{N,t}(s,s_0)=A\left(\frac{s}{N}+t\right)\Psi_{N,t}(s,s_0).
\end{gather*}
We assume that $A(u)$, $u\in\R^+$ has eigenvalues with strictly negative real part and 
consider as corresponding locally stationary approximation $\tilde{Y}_{u}$ the process
\begin{align}\label{eq:seqlimLDstatespacesolution}
\begin{aligned}
\tilde{Y}_{u}(t)&= \int_\R g(u,t-s)L(ds), \text{ with kernel function}\\
g(u,t-s)&= \mathbb{1}_{\{t-s\geq0\}} B(u)' e^{A(u)(t-s)}C\left(u \right),
\end{aligned}
\end{align}
where $e^{A(u)(t-s)}$ is the solution to the IVP
\begin{gather*}
\Psi_{u}(s_0,s_0)=\mathbf{1}_p,\qquad
\frac{d}{ds}\Psi_{u}(s,s_0)=A\left(u\right)\Psi_{u}(s,s_0).
\end{gather*}

\begin{Proposition}[{\cite[Proposition 4.11]{BS2021}}]
Let $Y_N$ be a sequence of time-varying L\'evy-driven state space models as given in (\ref{eq:seqtvLDstatespacesolution}). 
Then, $Y_N$ is locally stationary, if 
\begin{enumerate}[label={(\alph*)}]
\item the coefficient functions $A,B$ and $C$ are continuous,
\item $\norm{B(s)}<\infty$ for all $s\in\R$, $\sup_{s\in\R}\norm{C(s)}<\infty$ and
\item $\norm{\Psi_{N,t}(0,s)}\leq F_t(s)$ for some function $F_t\in L^2(\R^-)$ and all $N\in\N$, $t\in\R$.
\end{enumerate}
In particular, $(c)$ holds if $Y_N$ is uniformly exponential stable with the same $\gamma$ and $\lambda$ for all $N\in \N$, since then
\begin{gather*}
\norm{\Psi_{N,t}(0,s)}\leq F_t(s)=\gamma e^{\lambda s}\in L^2(\R^-).
\end{gather*}
The corresponding limiting local kernel function is given by $g(u,s)$ as defined in (\ref{eq:seqlimLDstatespacesolution}).
\end{Proposition}

\begin{Proposition}\label{proposition:LStvldsp}
Let $Y_N$ be a sequence of time-varying L\'evy-driven state space models as given in (\ref{eq:seqtvLDstatespacesolution}). 
Then, $\tilde{Y}_{u}$ as given in (\ref{eq:seqlimLDstatespacesolution}) is a locally stationary approximation of $Y_N$ for $p=2$, if 
\begin{enumerate}[label={(\alph*)}]
\item the coefficient functions $A,B$ and $C$ are Lipschitz with constants $L_A,L_B$ and $L_C$, respectively,
\item $\sup_{s\in\R}\norm{B(s)}<\infty$ and $\sup_{s\in\R}\norm{C(s)}<\infty$, 
\item $\norm{  \Psi_{N,t}(0,s)}\leq F(s)$ such that $s\mapsto |s|F(s) \in L^1(\R^-)\cap L^2(\R^-)$ for all $N\in\N$ and $t\in \R$,
\item $\norm{\Psi_{N,t}(0,s) -\Psi_{t}(0,s)}\leq \frac{\tilde F(s)}{N}$ such that $\tilde F(s)\in L^1(\R^-)\cap L^2(\R^-)$ for all $t\in\R$,
\item $\norm{\Psi_{u}(0,s)}\leq G(s)$ such that $G(s) \in L^1(\R^-)\cap L^2(\R^-)$ for all $u\in \R^+$,
\item $\norm{\Psi_u(0,s)-\Psi_v(0,s)}\leq \abs{u-v} \tilde G(s)$ such that $\tilde G(s)\in L^1(\R^-)\cap L^2(\R^-)$.
\end{enumerate}
If we additionally assume that $\int_{\R}x^4\nu(dx)<\infty$ and $|s|F(s),\tilde F(s),G(s),\tilde G(s)\in L^1(\R^-)\cap L^4(\R^-)$, then $\tilde{Y}_{u}$ is also a locally stationary approximation of $Y_N$ for $p=4$.
\end{Proposition}

\begin{proof}
It is clear that $\tilde{Y}_{u}$ is stationary for all $u\in\R^+$. From \cite[Theorem 3.5]{FS2013} it follows, that $\tilde{Y}_{u}$ is ergodic. Let $S_B=\sup_{s\in\R}\norm{B(s)}<\infty$, $S_C=\sup_{s\in\R}\norm{C(s)}<\infty$, $\mu_L= \gamma+\int_{|x|>1}x\nu(dx)$ and $\Sigma_L=\Sigma+\int_{\R}x^2\nu(dx)$. First, we note that
\begin{gather*}
\Lnorm{\tilde{Y}_{u}(t)}^2\leq \Sigma_LS_BS_C \int_{-\infty}^0G(s)^2ds+\mu_LS_B^2S_C^2\left(\int_{-\infty}^0G(s)ds\right)^2<\infty.
\end{gather*}
Now, for $u,v\in\R^+$
\begin{align*}
&\norm{ \tilde{Y}_{u}(t)-\tilde{Y}_{v}(t)}_{L^2}^2\\
&=\Sigma_L\int_\R\left( \mathbb{1}_{\{t-s\geq0\}} B(u)' \Psi_{u}(0,-(t-s))C\left(u \right) -\mathbb{1}_{\{t-s\geq0\}} B(v)' \Psi_{v}(0,-(t-s))C\left(v \right) \right)^2ds\\
&\quad +\mu_L^2 \Bigg(  \int_\R \mathbb{1}_{\{t-s\geq0\}} B(u)' \Psi_{u}(0,-(t-s))C\left(u \right)
 -\mathbb{1}_{\{t-s\geq0\}} B(v)' \Psi_{v}(0,-(t-s))C\left(v \right)ds\Bigg)^2\\
&=:P_1+P_2.
\end{align*}
For $P_1$ we obtain
\begin{align*}
P_1=& \Sigma_L\int_{-\infty}^{t}\abs{B(u)'\Psi_{u}(0,-(t-s))C(u)-B(v)'\Psi_{v}(0,-(t-s))C(v)}^2ds\\
=&\Sigma_L\int_{-\infty}^{t}\!\Bigg( B(u)'\Psi_{u}(0,-(t-s))C(u)\!-\!B(v)'\Psi_{u}(0,-(t-s))C(u)\!+\!B(v)'\Psi_{u}(0,-(t-s))C(u)\\
&-B(v)'\Psi_{v}(0,-(t-s))C(u)+B(v)'\Psi_{v}(0,-(t-s))C(u)-B(v)'\Psi_{v}(0,-(t-s))C(v)\Bigg)^2 ds\\
\leq& 3\Sigma_L\!\!\int_{-\infty}^{0}\! \!\!\bigg(\! B(u)'\Psi_{u}(0,s)C(u)\!-\!B(v)'\Psi_{u}(0,s)C(u)\! \bigg)^2\\
&\!\!+\!\bigg(\! B(v)'\Psi_{u}(0,s)C(u)\!-\!B(v)'\Psi_{v}(0,s)C(u)\!\bigg)^2 \!\!+\!\bigg(B(v)'\Psi_{v}(0,s)C(u)\!-\!B(v)'\Psi_{v}(0,s)C(v)\bigg)^2 ds.
\end{align*}
Now, 
\begin{align*}
\bigg( B(u)'\Psi_{u}(0,s)C(u)-B(v)'\Psi_{u}(0,s)C(u) \bigg)^2 &\leq  \norm{ B(u)'-B(v)'}^2 \norm{\Psi_{u}(0,s)C(u) }^2\\
&\leq  S_C^2L_B^2 (u-v)^2 G(s)^2.
\end{align*}
Analogously we obtain,
\begin{gather*}
\big(B(v)'\Psi_{v}(0,s)C(u)-B(v)'\Psi_{v}(0,s)C(v) \big)^2\leq S_B^2 L_C^2  (u-v)^2 G(s)^2
\end{gather*}
and
\begin{align*}
\big( B(v)'\Psi_{u}(0,s)C(u)-B(v)'\Psi_{v}(0,s)C(u) \big)^2 &\leq S_B^2 S_C^2 (u-v)^2 \norm{\Psi_{u}(0,s)-\Psi_{v}(0,s)}^2\\
&\leq S_B^2 S_C^2 (u-v)^2 \tilde G(s)^2,
\end{align*}
such that
\begin{gather*}
P_1 \leq 3\Sigma_L (u-v)^2 \int_{-\infty}^{0}\left( \left(S_C^2L_B^2+ S_B^2 L_C^2\right) G(s)^2+ S_B^2 S_C^2 \tilde{G}(s)^2 \right)ds.
\end{gather*}
For $P_2$ it holds
\begin{align*}
P_2&\leq \mu_L^2 \bigg( \int_{-\infty}^{0} \abs{ B(u)'\Psi_{u}(0,s)C(u)-B(v)'\Psi_{u}(0,s)C(u) }\\
&\quad+\! \abs{ B(v)'\Psi_{u}(0,s)C(u)\!-\!B(v)'\Psi_{v}(0,s)C(u)}\!+\!\abs{B(v)'\Psi_{v}(0,s)C(u)\!-\!B(v)'\Psi_{v}(0,s)C(v) } ds\!\bigg)^2\\
&\leq \mu_L^2 (u-v)^2 \bigg( \int_{-\infty}^0\left(S_CL_B+ S_B L_C\right) G(s)+ S_B S_C \tilde G(s)ds\bigg)^2.
\end{align*}
Overall, we obtain
\begin{gather*}
\norm{ \tilde{Y}_{u}(t)-\tilde{Y}_{v}(t)}_{L^2} \leq
D_1 \abs{u-v}
\end{gather*}
for a constant $D_1>0$. Second,
\begin{align*}
\Lnorm{Y_N(t)-\tilde{Y}_t(Nt)}^2=&
\Sigma_L \int_\R \bigg( \mathbb{1}_{\{Nt-s\geq0\}} B(t)'\Psi_{N,t}(0,-(Nt-s))C\left(\frac{-(Nt-s)}{N}+t \right) \\
&\qquad\quad - \mathbb{1}_{\{Nt-s\geq0\}} B(t)' \Psi_{t}(0,-(Nt-s))C\left(t \right)\bigg)^2ds\\
&+\mu_L^2 \Big(\int_\R\mathbb{1}_{\{Nt-s\geq0\}} B(t)' \Psi_{N,t}(0,-(Nt-s))C\left(\frac{-(Nt-s)}{N}+t \right)\\
& \qquad-\mathbb{1}_{\{Nt-s\geq0\}} B(t)' \Psi_{t}(0,-(Nt-s))C\left(t \right) ds \Big)^2\\
=:& P_3+P_4.
\end{align*}
For $P_3$ we obtain
\begin{align*}
P_3
=&\Sigma_L\int_{-\infty}^{0}\!\!\Bigg( B(t)'\Bigg(\Psi_{N,t}(0,s) C\left(\frac{s}{N}\!+\!t \right)\!-\!\Psi_{N,t}(0,s)C(t)
\!+\!\Psi_{N,t}(0,s)C(t)\!-\!\Psi_{t}(0,s)C(t)\Bigg)\Bigg)^2ds\\
\leq& 2\Sigma_L S_B^2\int_{-\infty}^{0}\!\! \bigg( \norm{ \Psi_{N,t}(0,s)C\left(\frac{s}{N}\!+\!t\right)\!-\!\Psi_{N,t}(0,s)C(t)}^2
\!+\!\norm{ \Psi_{N,t}(0,s)C(t)\!-\!\Psi_{t}(0,s)C(t)}^2\!\bigg)ds\\
\leq&  2\Sigma_L S_B^2\int_{-\infty}^{0} \norm{  \Psi_{N,t}(0,s) }^2\norm{C\left(\frac{s}{N}+t\right)-C(t)}^2ds\\
&+2\Sigma_L S_B^2\int_{-\infty}^{0}\norm{C(t)}^2  \norm{ \Psi_{N,t}(0,s) -\Psi_{t}(0,s)}^2ds\\
\leq& 2\Sigma_L S_B^2\int_{-\infty}^{0}\bigg( F(s)^2 \norm{C\left(\frac{s}{N} +t\right)-C(t)}^2\bigg)ds
+ 2\Sigma_L S_B^2 S_C^2 \int_{-\infty}^{0}\norm{\Psi_{N,t}(0,s) -\Psi_{t}(0,s)}^2ds \\
\leq& \frac{2\Sigma_L S_B^2L_C^2}{N^2}\int_{-\infty}^{0} F(s)^2 s^2 du + \frac{2\Sigma_L S_B^2 S_C^2}{N^2} \int_{-\infty}^{0} \tilde F(s)^2 ds,
\end{align*}
since $C(\cdot)$ is Lipschitz with constant $L_C$. For $P_4$ it holds
\begin{align*}
P_4\leq& \mu_L^2S_B^2 \left(\int_{-\infty}^0\! \norm{ \Psi_{N,t}(0,s)C\left(\frac{s}{N}\!+\!t\right)\!-\!\Psi_{N,t}(0,s)C(t)}
\!+\!\norm{ \Psi_{N,t}(0,s)C(t)\!-\!\Psi_{t}(0,s)C(t)}ds\right)^2\\
\leq& \frac{ \mu_L^2S_B^2}{N^2} \left(L_C \int_{-\infty}^{0} F(s) |s| ds + S_C^2 \int_{-\infty}^{0} \tilde F(s) ds \right)^2.
\end{align*}
All together
\begin{gather*}
\Lnorm{Y_N(t)-\tilde{Y}_t(Nt)}\leq
 \frac{D_2}{N}
\end{gather*}
for a constant $D_2>0$.\\
Since $\int_{\R}x^4\nu(dx)<\infty$, Lemma \ref{lemma:4thmoment} gives closed form expressions for $\lVert Y_N(t)-\tilde{Y}_t(Nt)\rVert_{L^4}^4$ and $\norm{ \tilde{Y}_{u}(t)-\tilde{Y}_{v}(t)}_{L^4}^4$. Then, calculations analogous to the steps for $p=2$ show that $\tilde{Y}_u$ is a locally stationary approximation of $Y_N$ for $p=4$.

\end{proof}

\begin{Corollary}\label{corollary:LStvldsp}
Let $Y_N$ be a sequence of time-varying L\'evy-driven state space models as given in (\ref{eq:seqtvLDstatespacesolution}). Then, $\tilde{Y}_{u}$ as given in (\ref{eq:seqlimLDstatespacesolution}) is a locally stationary approximation of $Y_N$ for $p=2$, if 
\begin{enumerate}[label={(\alph*)}]
\item[(a1)] the coefficient functions $A,B$ and $C$ are Lipschitz with constants $L_A,L_B$ and $L_C$, respectively,
\item[(b1)] $\sup_{s\in\R}\norm{B(s)}<\infty$ and $\sup_{s\in\R}\norm{C(s)}<\infty$, 
\item[(c1)] $\{A(t)\}_{t\in\R}$ commutes, i.e. $[A(t),A(s)]=0$ for all $s,t\in\R$,
\item[(d1)] $\norm{e^{\nu\int_s^0A(\frac{\tau}{N}+t)d\tau}}\leq \gamma e^{\nu s \lambda}$, with $\gamma,\lambda>0$ for all $\nu\in[0,1]$, $s<0$, $t\in\R$ and $N\in\N$, and
\item[(e1)] $\norm{e^{-\nu A(t)s}}\leq \tilde\gamma e^{\nu s \tilde\lambda}$, with $\tilde\gamma,\tilde\lambda>0$ for all $\nu\in[0,1]$, $s<0$ and $t\in\R$,
\end{enumerate}
or if 
\begin{enumerate}[label={(\alph*)}]
\item[(a2)] the coefficient functions $A,B$ and $C$ are Lipschitz with constants $L_A,L_B$ and $L_C$, respectively,
\item[(b2)] $\sup_{s\in\R}\norm{B(s)}<\infty$ and $\sup_{s\in\R}\norm{C(s)}<\infty$, 
\item[(c2)] $\{A(t)\}_{t\in\R}$ commutes, i.e. $[A(t),A(s)]=0$ for all $s,t\in\R$,
\item[(d2)] the eigenvalues $\lambda_j(t)$ of $A(t)$ for $j=1,\ldots,p$ satisfy $\sup_{t\in\R}\max_{j=1,\ldots,p}\mathfrak{Re}(\lambda_j(t))<0$ and
\item[(e2)] $A(t)$ is diagonalizable for all $t\in\R$.
\end{enumerate}
If additionally $\int_{\R}x^4\nu(dx)<\infty$, the conditions (a1)-(e1) and (a2)-(e2) are also sufficient for $\tilde{Y}_{u}$ to be a locally stationary approximation of $Y_N$ for $p=4$.
\end{Corollary}
\begin{proof}
It is enough to show show the conditions (c)-(f) of Proposition \ref{proposition:LStvldsp}. Since $[A(t),A(s)]=0$ for all $s,t\in\R$, it holds $\Psi_{N,t}(t_1,t_0)=e^{\int_{t_0}^{t_1}A(\frac{\tau}{N}+t)d\tau}$ and $\Psi_{t}(t_1,t_0)=e^{A(t)(t_1-t_0)}$. Assume now that (a1)-(e1) hold. We obtain
\begin{align*}
\norm{\Psi_{N,t}(0,s)}&=\norm{e^{\int_{s}^{0}A(\frac{\tau}{N}+t)d\tau}}\leq \gamma e^{ s \lambda} \text{, where }|s|\gamma e^{ s \lambda}\in L^1(\R^-)\cap L^4(\R^-) \text{ and}\\
\norm{\Psi_{u}(0,s)}&=\norm{e^{-A(u)s}}\leq \tilde\gamma e^{ s \tilde\lambda} \in L^1(\R^-)\cap L^4(\R^-).
\end{align*}
In order to show (d) and (f), we first note that (see \cite[page 238]{H2008})
\begin{align*}
e^A-e^B=\int_0^1 e^{\nu B}(A-B)e^{(1-\nu)A}d\nu.
\end{align*}
Then,
\begin{align*}
\norm{\Psi_{N,t}(0,s) -\Psi_{t}(0,s)}=&\norm{ e^{\int_{s}^{0}A(\frac{\tau}{N}+t)d\tau} - e^{-A(t)s} }\\
=& \norm{ \int_0^1 e^{\nu\int_{s}^0A\left(\frac{\tau}{N}+t\right)d\tau} \left(\int_{s}^0A\left(\frac{\tau}{N}+t\right)d\tau +A(t)s \right)e^{-(1-\nu)A(t)s} d\nu  }\\
\leq& \norm{ \int_{s}^0A\left(\frac{\tau}{N}+t\right)d\tau +A(t)s}  \int_0^1\norm{ e^{\nu\int_{s}^0A\left(\frac{\tau}{N}+t\right)d\tau}} \norm{ e^{-(1-\nu)A(t)s}} d\nu\\
\leq& \int_{s}^0 \norm{A\left(\frac{\tau}{N}+t\right) -A(t)} d\tau  \int_0^1 \gamma e^{\nu\lambda  s} \tilde\gamma e^{(1-\nu)\tilde\lambda  s} d\nu\\
\leq& \int_{s}^0 L_A\frac{|\tau|}{N}d\tau \gamma \tilde\gamma\int_0^1  e^{\nu\hat\lambda  s}  e^{(1-\nu)\hat\lambda  s} d\nu=\frac{ L_A \gamma \tilde\gamma}{2N} s^2 e^{\hat\lambda s},
\end{align*}
where $\hat\lambda=\min(\lambda,\tilde\lambda)>0$ and $\tilde F(s)= \frac{L_A \gamma \tilde\gamma}{2} s^2 e^{\hat\lambda s}\in L^1(\R^-)\cap L^4(\R^-)$. Finally,
\begin{align*}
\norm{\Psi_u(0,s)\!-\!\Psi_v(0,s)} =& \norm{e^{-A(u)s}\!-\!e^{-A(v)s} }
\leq |s| \norm{A(u)\!-\!A(v)}  \int_0^1 \norm{e^{-\nu A(u)s}} \norm{ e^{-(1-\nu)A(v)s}}d\nu\\
\leq& L_A |s| \abs{u-v} \int_0^1 \tilde\gamma e^{\nu \tilde\lambda  s}  \tilde\gamma e^{(1-\nu)\tilde\lambda  s} d\nu = \abs{u-v}  L_A \tilde\gamma^2 |s| e^{ \tilde\lambda  s},
\end{align*}
where $\tilde G(s)=L_A \tilde\gamma^2 |s| e^{ \tilde\lambda  s}\in L^1(\R^-)\cap L^4(\R^-)$.\\
For the second part we assume (a2)-(e2) to hold. It suffices to show that (d1) and (e1) hold. Since $\{A(t)\}$ is mutually commutative and (e2) holds, $\{A(t)\}$ is simultaneously diagonalizable (see \cite[Theorem 1.3.12]{HJ1990}). Therefore, there exists a regular matrix $S$ such that $SA\left(\frac{s}{N}+t\right)S^{-1}=diag(\lambda_1\left(\frac{s}{N}+t\right),\ldots, \lambda_p\left(\frac{s}{N}+t\right) )=D\left(\frac{s}{N}+t\right)$. In the following, we set $-\beta=\sup_{t\in\R}\max_{j=1,\ldots,p}$ $\mathfrak{Re}(\lambda_i(t))<0$ and $\norm{\cdot}_{Spec}$ denotes the operator norm induced by the euclidean norm. Then,
\begin{align*}
\norm{ e^{\nu \int_s^0A\left(\frac{\tau}{N}+t\right)d\tau} }=& \norm{ e^{\nu \int_s^0 S^{-1} D\left(\frac{\tau}{N}+t\right)Sd\tau} }= \norm{  S^{-1}e^{\nu \int_s^0 D\left(\frac{\tau}{N}+t\right)d\tau}S }\\
\leq& \norm{S^{-1}} \norm{S }\norm{e^{\nu \int_s^0 D\left(\frac{\tau}{N}+t\right)d\tau} }\leq C\norm{e^{\nu \int_s^0 D\left(\frac{\tau}{N}+t\right)d\tau} }_{Spec}\\
=&C \max\left\{\!\sqrt{\mu}, \mu\in\sigma\!\left(\left(e^{\nu \int_s^0 D\left(\frac{\tau}{N}+t\right)d\tau}\right)'\left(e^{\nu \int_s^0 D\left(\frac{\tau}{N}+t\right)d\tau}\right)\right)\! \right\}\\
=&C \max\left\{\!\sqrt{\mu}, \mu\in\sigma\!\left(e^{\nu \int_s^0 D\left(\frac{\tau}{N}+t\right)'+D\left(\frac{\tau}{N}+t\right)d\tau}\right)\! \right\}
\!=\! C\!\! \max_{j=1,\ldots,p}\!\! \sqrt{e^{2\nu \int_s^0\mathfrak{Re} \left(\lambda_j\left(\frac{\tau}{N}+t\right) \right)d\tau }}\\
=& C \max_{j=1,\ldots,p}e^{\nu \int_s^0\mathfrak{Re} \left(\lambda_j\left(\frac{\tau}{N}+t\right) \right)d\tau }\leq C e^{\nu s\beta}
\end{align*}
for a constant $C>0$. Analogously we obtain 
\begin{gather*}
\norm{e^{-\nu A(t)s}}\leq C\norm{e^{-\nu D(t)s} }_{Spec} \leq C e^{\nu s \beta}.
\end{gather*}
\end{proof}

The next theorem establishes a number of asymptotic results for different sample moments and observations from a sequence of time-varying L\'evy-driven state space models based on our results from Section \ref{sec3}.\\

\begin{Theorem}\label{theorem:statisticfortvlvsp}
Let $Y_N$ be a sequence of time-varying L\'evy-driven state space models as given in (\ref{eq:seqtvLDstatespacesolution}) and assume that one of the following sets of conditions holds: 
\begin{enumerate}[label={(\alph*)}]
\item[(A1)] conditions (a)-(f) from Proposition \ref{proposition:LStvldsp}.
\item[(A2)] conditions (a1)-(e1) from Corollary \ref{corollary:LStvldsp}.
\item[(A3)] conditions (a2)-(d2) from Corollary \ref{corollary:LStvldsp}. 
\end{enumerate}
Then, we obtain the following results:
\begin{enumerate}[label={(\alph*)}]
\item For observations sampled according to \hyperref[observations:O1]{(O1)} or \hyperref[observations:O2]{(O2)}, we obtain for all $u\in\R^+$ and $k\in\N_0$
\begin{align}
\frac{\delta_N}{b_N} \sum_{i=-m_N}^{m_N}K\left(\frac{\tau_i^N-u}{b_N} \right)Y_N(\tau_i^N) &\overset{L^1}{\underset{N\rightarrow\infty}{\longrightarrow}}E[\tilde{Y}_u(0)]\text{ and for \hyperref[observations:O1]{(O1)} also}\label{equation:tvldspsamplemean}\\
\frac{\delta_N}{b_N} \sum_{i=-m_N}^{m_N}K\left(\frac{\tau_i^N-u}{b_N} \right)Y_N(\tau_i^N)Y_N\left(\tau_i^N+\frac{k}{N}\right)&\overset{L^1}{\underset{N\rightarrow\infty}{\longrightarrow}}E[\tilde{Y}_u(0)\tilde{Y}_u(k)]\label{equation:tvldspsamplecov}
\end{align}
with $\tilde{Y}_{u}$ as given in (\ref{eq:seqlimLDstatespacesolution}). 
\item If we additionally assume that either (A2) or (A3) hold and
\begin{enumerate}[label={(\alph*)}]
\setcounter{enumi}{2}
\item[(B)] $\int_{|x|>1}|x|^{2+\varepsilon}\nu(dx)<\infty$ for some $\varepsilon>0$,
\end{enumerate}
then the convergence in (\ref{equation:tvldspsamplecov}) also holds for \hyperref[observations:O2]{(O2)}. 
\item If in addition to (b) it holds that
\begin{enumerate}[label={(\alph*)}]
\setcounter{enumi}{2}
\item[(C)] $E[L(1)]=0$,
\item[(D)] $\sqrt{m_N}b_N\rightarrow0$, as $N\rightarrow\infty$ and
\item[(E)] $\sigma(u)^2>0$, where $\sigma(u)^2=\begin{cases}\frac{1}{2}E[\tilde{Y}_u(0)^2]+\sum_{h=1}^\infty E[\tilde{Y}_u(0)\tilde{Y}_u(h)] , &\text{ if \hyperref[observations:O1]{(O1)} holds,}\\
E[\tilde{Y}_u(0)^2],&\text{ if \hyperref[observations:O2]{(O2)} holds,} \end{cases}$
\end{enumerate}
then $\sigma(u)^2<\infty$ and
\begin{gather}\label{equation:tvldspcltsamplemean}
\sqrt{\frac{\delta_N}{b_N}} \sum_{i=-m_N}^{m_N}K_{rect}\left(\frac{\tau_i^N-u}{b_N}\right)Y_N(\tau_i^N)\overset{d}{\underset{N\rightarrow\infty}{\longrightarrow}} \sigma(u)^2 Z,
\end{gather}
where $Z$ is a standard normally distributed random variable and $K_{rect}$ as defined in (\ref{equation:rectangularkernel}). 
\item[(d)] If in addition to (c) it holds that
\begin{enumerate}[label={(\alph*)}]
\setcounter{enumi}{5}
\item[(F)] $\int_{|x|>1}|x|^{4+\varepsilon}\nu(dx)<\infty$ for some $\varepsilon>0$ \text{ and}
\item[(G)] 
$\tilde{\sigma}(u)^2>0$, where \\
$\tilde{\sigma}(u)^2=\begin{cases}\frac{1}{2}E[\tilde{Y}_u(0)^2\tilde{Y}_u(k)^2]+\sum_{h=1}^\infty E[\tilde{Y}_u(0)\tilde{Y}_u(k)\tilde{Y}_u(h)\tilde{Y}_u(h+k)] , &\text{ if \hyperref[observations:O1]{(O1)} holds,}\\
E[\tilde{Y}_u(0)^2\tilde{Y}_u(k)^2],&\text{ if \hyperref[observations:O2]{(O2)} holds,} \end{cases}$
\end{enumerate}
then $\tilde{\sigma}(u)^2<\infty$ and
\begin{gather}\label{equation:tvldspcltsamplecov}
\begin{gathered}
\sqrt{\frac{\delta_N}{b_N}}\! \sum_{i=-m_N}^{m_N} K_{rect} \left(\frac{\tau_i^N-u}{b_N}\right)\left(Y_N(\tau_i^N)Y_N(\tau_i^N+\frac{k}{N})-Cov(\tilde{Y}_u(0)\tilde{Y}_u(k))\right)\\
\overset{d}{\underset{N\rightarrow\infty}{\longrightarrow}} \tilde{\sigma}(u)^2Z.
\end{gathered}
\end{gather}
\end{enumerate}
\end{Theorem}
\begin{proof}
We first note that \cite[Corollary 3.4]{CS2018} implies that $\tilde{Y}_u$ is $\theta$-weakly dependent. The convergences in (\ref{equation:tvldspsamplemean}) and (\ref{equation:tvldspsamplecov}) then follow analogous to the proof of Theorem \ref{theorem:statisticforcar1}. The proof of \cite[Theorem 3.36]{CSS2020} together with (A2) or (A3) imply that the $\theta$-coefficients of $\tilde{Y}_u$ are exponentially decaying.
Therefore, we also obtain the convergences in (\ref{equation:tvldspcltsamplemean}) and (\ref{equation:tvldspcltsamplecov}).
\end{proof}

\begin{Remark}
In ongoing research, the developed asymptotic results from Section \ref{sec3} are applied to establish a comprehensive asymptotic theory for M-estimators of contrast functions that are based on observations from a sequence of time-varying L\'evy-driven state space processes. The contrast functions under investigation are of finite and infinite memory and satisfy suitable regularity conditions that are in the spirit of the function classes $\mathcal{L}_\infty^{p,q}$ and $\mathcal{L}_{k+1}$ from Section \ref{sec2-3} and \ref{sec2-4}. \\
In particular, this includes consistency and asymptotic normality results for a localized least square estimator for time-varying CAR(1) processes as well as consistency results for a localized quasi maximum-likelihood and a localized Whittle estimator for time-varying L\'evy-driven state space models.
\end{Remark}

\section*{Acknowledgements}
The second author was supported by the scholarship program of the Hanns-Seidel Foundation, funded by the Federal Ministry of Education and Research.

\bibliographystyle{acm}

\end{document}